\documentclass[12pt]{amsart} % article/book/letter/amsart

\usepackage[utf8]{inputenc}
\usepackage[T1]{fontenc}	% accents
\usepackage[french,english]{babel}

\usepackage{lmodern}			% police de caractères
\usepackage{newtxtext}

\usepackage[top=3cm, bottom=3cm, left=3.6cm, right=3.6cm]{geometry}

\usepackage{graphicx}	% insertion des images / graphiques

\usepackage{tikz}
\usetikzlibrary{patterns}

\usepackage[plainpages=false, colorlinks, linkcolor=bleuFonce, citecolor=rougeFonce, urlcolor=vertFonce, breaklinks]{hyperref}
\usepackage{placeins}
\usepackage{float} % pour mettre H dans les figures

\usepackage{color}
\definecolor{vertFonce}{rgb}{0,0.5,0}
\definecolor{numLignes}{rgb}{0.17,0.57,0.7}	%{43,145,175}
\definecolor{gris}{rgb}{0.5,0.5,0.5}
\definecolor{grisFonce}{rgb}{0.2,0.2,0.2}
\definecolor{orange}{rgb}{1,0.65,0.31}		%{255,167,79}
\definecolor{orangeFonce}{rgb}{1,0.4,0}
\definecolor{bleuFonce}{rgb}{0,0,0.4}
\definecolor{rougeFonce}{rgb}{0.3,0,0}
\definecolor{rougeWord}{rgb}{0.5,0,0}
\definecolor{vertClair}{rgb}{0.8,1,0.8}
\definecolor{rougeClair}{rgb}{1,0.5,0.5}

\usepackage{pict2e}
\usepackage{multido}

\usepackage{amsfonts,amssymb,amsthm,amsmath}
\usepackage{amsaddr}		% author address (option [foot])
\usepackage{dsfont}			% \mathds symbols
\usepackage{mathrsfs}
\usepackage{yfonts}

\newtheorem{lem}{Lemma}[section]
\newtheorem{thm}{Theorem}[section]
\newtheorem{cor}{Corollary}[section]
\newtheorem{prop}{Proposition}[section]

\newtheorem{remark}{Remark}[section]

\newenvironment{system}{%
	\equation\left\{\ \begin{aligned}%
}{%
	\end{aligned} \right. \endequation%
}
\newenvironment{system*}{%
	\equation\nonumber\left\{\ \begin{aligned}
}{%
	\end{aligned} \right. \endequation%
}

\newcommand		{\N}		{\mathbb N}			% naturels {0,1,2,3,...}
\newcommand		{\RR}		{\mathbb R}			% réels
\newcommand		{\R}		{\RR}
\newcommand		{\Rd}		{\R^d}
\newcommand		{\Rdd}		{\R^{2d}}
\newcommand		{\hd}		{h^d}

\renewcommand	{\SS}		{\mathds S}			% sphère unité
\newcommand		{\cH}		{\mathcal H}		% Hilbert
\newcommand		{\cI}		{\mathcal I}

\newcommand		{\cP}		{\mathcal P}		% Density operators
\newcommand		{\opP}		{\cP}				% Opérateurs de trace 1
\renewcommand	{\L}		{\mathcal L}		% Semiclassical schatten or law random var
\newcommand     {\cW}		{\mathcal W}		% Quantum Sobolev
\renewcommand	{\P}		{\mathscr P}		% Proba density

\newcommand		\sfT		{\mathsf T}			% Opérateur de transport

\newcommand		{\lt}			{\left}				%
\newcommand		{\rt}			{\right}			%
\renewcommand	{\(}			{\lt(}
\renewcommand	{\)}			{\rt)}
\newcommand		{\set}[1]		{\lt\{#1\rt\}}
\newcommand		{\floor}[1]		{\lt\lfloor{#1}\rt\rfloor}
\newcommand		{\bangle}[1]	{\lt\langle #1\rt\rangle}
\newcommand		{\com}[1]		{\lt[{#1}\rt]}		% commutator

\newcommand		{\n}[1]			{\lt|{#1}\rt|}
\newcommand		{\nrm}[1]		{\lt\|{#1}\rt\|}
\newcommand		{\snrm}[1]		{\lVert #1\rVert}

\newcommand		{\Nrm}[2]		{\nrm{#1}_{#2}}
\newcommand		{\sNrm}[2]		{\snrm{#1}_{#2}}

\newcommand		{\indic}	{\mathds{1}}		% indicatrice
\renewcommand		{\d}		{\mathop{}\!\mathrm{d}}		% différentielle
\newcommand			{\dpt}		{\partial_t}

\newcommand			{\Dx}		{\nabla_x}
\newcommand			{\Dv}		{\nabla_\xi}
\newcommand			{\conj}[1]	{\overline{#1}}		% conjugaison complexe
\newcommand			{\Id}		{\mathrm{Id}}		% Identity operator

\DeclareMathOperator{\re}		{Re}				% partie réelle
\DeclareMathOperator{\tr}		{Tr}				% Trace

\renewcommand	{\Re}[1]		{\re\!\( #1 \)}		% partie réelle
\newcommand		{\Tr}[1]		{\tr\!\( #1 \)} 	% Trace

\newcommand		{\intd}			{\int_{\R^d}}
\newcommand		{\intdd}		{\int_{\R^{2d}}}
\newcommand		{\iintd}		{\iint_{\R^{2d}}}

\newcommand		{\jj}			{\mathrm{j}}	% j\in\N
\newcommand		{\init}			{\mathrm{in}}
\newcommand		{\loc}			{\mathrm{loc}}

\newcommand		{\eps}			{\varepsilon}
\newcommand		{\Eps}			{\mathcal{E}}
\newcommand		{\cC}			{\mathcal{C}}

\usepackage{braket}
\newcommand		{\Inprod}[2]	{\Braket{#1 | #2}}

\newcommand		{\op}		{\boldsymbol{\rho}}	% density operator

\newcommand		{\opm}		{\boldsymbol{m}}	% operator
\newcommand		{\opmu}		{\boldsymbol{\mu}}	% operator

\newcommand		{\opgam}	{\boldsymbol{\gamma}}% couplage quantique
\newcommand		{\opc}		{\boldsymbol{c}}

\DeclareMathOperator{\W}	{W}
\newcommand		{\Wh}		{\W_{2,\hbar}}		% Semiclassical Wasserstein
\newcommand		{\MKh}		{\W_{2,\hbar}}		% Quantum Wasserstein

\newcommand		{\tildop}		{\,\tilde{\!\op}}	% Wick quantization
\newcommand		{\ttildop}		{\,\tilde{\tilde{\!\op}}}	% Wick quantization

\newcommand		{\opp}		{\boldsymbol{p}}
\newcommand		{\opz}		{\mathbf{z}}%{\boldsymbol{z}}

\newcommand		{\Dh}		{\boldsymbol{\nabla}}	% quantum gradient
\newcommand		{\DDh}		{\boldsymbol{\Delta}}	% quantum Laplacian
\newcommand		{\Dhx}[1]	{\Dh_{\!x} #1}			% quantum d_x
\newcommand		{\Dhv}[1]	{\Dh_{\!\xi} #1}		% quantum d_xi
\newcommand		{\Dhxj}[1]	{\Dh_{\!x_\jj} #1}		% quantum d_x_j
\newcommand		{\Dhvj}[1]	{\Dh_{\!\xi_\jj} #1}	% quantum d_xi_j

\newcommand			{\h}	{\mathfrak{h}}		% Hilbert space (L^2(\Rd))

\title[\textsc{Quantum Optimal Transport}]{\Large Quantum Optimal Transport and Weak Topologies}
\author[\textsc{L.~Lafleche}]{\large\textsc{Laurent Lafleche}}
\address{Institut Camille Jordan, UMR 5208 CNRS \\\& Université Claude Bernard Lyon 1, France}
\curraddr{\textsc{Unit\'e de Math\'ematiques pures et appliqu\'ees, \'Ecole Normale Supérieure de Lyon, Lyon, France}}
\email{laurent.lafleche@ens-lyon.fr}

\begin{document}

\begin{abstract}
	Several extensions of the classical optimal transport distances to the quantum setting have been proposed. In this paper, we investigate the pseudometrics introduced by Golse, Mouhot and Paul in [Commun Math Phys 343:165--205, 2016] and by Golse and Paul in [Arch Ration Mech Anal 223:57--94, 2017]. These pseudometrics serve as a quantum analogue of the Monge--Kantorovich--Wasserstein distances of order $2$ on the phase space. We prove that they are comparable to negative Sobolev norms up to a small term due to a positive "self-distance" in the semiclassical approximation, which can be bounded above using the Wigner--Yanase skew information. This enables us to improve the known results in the context of the mean-field and semiclassical limits by requiring less regularity on the initial data.
\end{abstract}

\begingroup
\def\uppercasenonmath#1{} % this disables uppercasing title
\let\MakeUppercase\relax % this disables uppercasing authors
\maketitle
\endgroup

%\textbf{Keywords}: Hartree equation, Hartree-Fock equation, Vlasov equation, Coulomb interaction, gravitational interaction, semiclassical limit.

%----------  Table of Contents  ----------
\renewcommand{\contentsname}{\centerline{Table of Contents}}
\setcounter{tocdepth}{2}	% profondeur du commentaire
\tableofcontents
%\addcontentsline{toc}{chapter}{Table of Contents}

%% ********************  Contenu  ********************

%----------  Introduction  ----------
\section{Introduction}
	
	There has been a recent surge of interest in the analogue of quantum optimal transportation within the realm of quantum mechanics, as well as in non-commutative settings in general. The objective of this paper is to compare the pseudometrics introduced in~\cite{golse_mean_2016, golse_schrodinger_2017} which are a quantum analogue of the (Monge--Kantorovich)--Wasserstein distances, with the quantum analogue of negative Sobolev norms for a wide variety of operators. These tools were used in particular to study the mean-field and classical limit from many body quantum mechanics in \cite{golse_mean_2016, golse_schrodinger_2017, lafleche_propagation_2019, golse_semiclassical_2021, lafleche_global_2021, porat_magnetic_2022}. Furthermore, additional properties of these pseudometrics were investigated in \cite{golse_wave_2018, caglioti_quantum_2020, golse_optimal_2022, caglioti_towards_2022}. 
	
	The pseudometrics are created by analogy with the classical Kantorovich formulation of the optimal transport problem with quadratic cost, and we refer for example to \cite{villani_topics_2003, santambrogio_optimal_2015} for the definition of the classical optimal transport problem and distances. The use of the word "pseudometric" is due to the fact that these quantum analogues of the optimal transport distances do not strictly qualify as distances. This is because the distance between an element and itself is strictly positive, although vanishing in the semiclassical limit. It can be seen as an uncertainty principle. Our first main result in this paper will be an estimation of this "self-distance".
	
	Other contexts and other formulations of the classical optimal transport distances have brought several definitions of quantum analogues of optimal transport distances. One of the first proposals to measure the optimal transport distance between quantum states can be found in~\cite{zyczkowski_monge_1998}, which involves utilizing the classical Wasserstein distance of the Husimi transforms of the operators. Other Wasserstein-type functionals were proposed to study the behavior of quantum Markov processes by Agredo in~\cite{agredo_wasserstein-type_2013} taking inspiration from the Kantorovich--Rubinstein duality, in~\cite{carlen_analog_2014}, with the gradient flow point of view and following the ideas of Benamou and Brenier. Proposals have also been made in~\cite{chakrabarti_quantum_2019, de_palma_quantum_2021, de_palma_quantum_2021-1, kiani_learning_2022, toth_quantum_2022} with some applications to quantum machine learning, where the cost function is refined to improve some properties of the distance. Let us in particular give a slightly longer comment on the pseudometric built by De Palma and Trevisan in~\cite{de_palma_quantum_2021-1}, which is also inspired by the  Kantorovich formulation of optimal transport, but with a transposition in the cost. In the context of this pseudometric, the "self-distance" is actually proved to be exactly equal to the Wigner--Yanase information.
	
\subsection{Quantum optimal transport pseudometrics}

	We now introduce more precisely the pseudometrics as well as other definitions and notations used in this paper. The pseudometrics are defined between probability densities on the phase space, whose set will be denoted by $\P(\Rdd)$, and density operators, whose set is defined as 
	\begin{equation*}
		\cP(\h) := \set{\op\in \L^\infty(\h) : \op \geq 0 \text{ and } \hd \Tr{\op} = 1}
	\end{equation*}
	where $\L^\infty(\h)$ is the set of bounded operators acting on $\h = L^2(\Rd)$. Notice that our normalization $\hd\Tr{\op} = 1$ is different from the one used in \cite{golse_mean_2016, golse_schrodinger_2017}, but one can easily pass from one to the other by considering $R := \hd\op$, which is of trace $1$. We also define the space of densities with $n$ bounded moments
	\begin{align}\label{eq:def_Pn}
		\P_n(\Rdd) &:= \set{f\in\P(\Rdd), \intdd \n{z}^n f(\d z) < \infty}
		\\\label{eq:def_cPn}
		\cP_n(\h) &:= \set{\op\in\cP(\h), \hd\Tr{\sqrt{\op} \(\n{x}^n+\n{\opp}^n\)\sqrt{\op}} < \infty}
	\end{align}
	where the momentum operator is defined by
	\begin{equation*}
		\opp := -i\hbar\nabla	
	\end{equation*}
	and for any operator $A$, its modulus is defined by $\n{A} = \sqrt{A^*A}$ where $A^*$ is the adjoint of $A$. In particular, we will write $\n{\opp}^n = \hbar^n\(-\Delta\)^{n/2}$.
	
	To define the pseudometrics, we first introduce the notion of coupling. If $f\in\P(\Rdd)$ and $\op\in \cP(\h)$, the set of couplings of $f$ and $\op$, denoted by $\cC(f,\op)$ is defined as the set of operator valued measures $\opgam = \opgam(z)$ such that
	\begin{equation*}
		\intdd \opgam(z)\d z = \op \quad \text{ and } \quad \hd\Tr{\opgam(z)} = f(z) 
	\end{equation*}
	where the second identity should in general be interpreted in the sense of distributions. The semiclassical optimal transport pseudometric between $f$ and $\op$ is then defined in~\cite{golse_schrodinger_2017, golse_optimal_2022} as\footnote{Again, the $\hd$ appears in the definition because of the choice of normalization. This is indeed the same quantity as the quantity $\mathfrak{d}$ in~\cite{golse_optimal_2022} since if $R = \hd\,\op$, then $\Wh(f,\op) = \mathfrak{d}(f,R)$.}
	\begin{equation*}
		\Wh(f,\op)^2 := \inf_{\opgam\in\cC(f,\op)} \hd\intdd \Tr{\opgam(z)^{1/2}\,\opc(z)\,\opgam(z)^{1/2}} \d z
	\end{equation*}
	with $z=(y,\xi)$ and $\opc(z) = \n{y-x}^2 + \n{\xi-\opp}^2$. We also allow to invert the position of the function and the operator in the notation by defining $\Wh(\op,f) := \Wh(f,\op)$. Similarly, if $\op_2\in\cP(\h)$ is another density operator, then the set of couplings between $\op$ and $\op_2$, denoted by $\cC(\op,\op_2)$, is defined as the set of operators $\opgam\in\cP(\h\otimes\h)$ such that
	\begin{equation*}
		\hd \tr_1(\opgam) = \op_2 \quad \text{ and } \quad \hd\tr_2(\opgam) = \op,
	\end{equation*}
	where $\tr_1$ denotes the partial trace with respect to the first variable and $\tr_2$ the partial trace with respect to the second variable. The quantum optimal transport pseudometric between $\op$ and $\op_2$ is then defined in~\cite[Definition~2.2]{golse_mean_2016} and \cite{golse_optimal_2022} as
	\begin{equation}\label{eq:def_MKh}
		\MKh(\op,\op_2)^2 := \inf_{\opgam\in\cC(\op,\op_2)} h^{2d}\Tr{\opgam^{1/2}\,\opc\,\opgam^{1/2}}
	\end{equation}
	where $\opc = \n{x-x_2}^2 + \n{\opp-\opp_2}^2$ is an unbounded operator on $\h\otimes\h$. These quantities are not distances since, as proved in \cite[Theorem~2.4]{golse_schrodinger_2017} and \cite[Theorem~2.3]{golse_mean_2016}, they satisfy for any $(f,\op,\op_2)\in \P(\Rdd)\times\cP(\h)^2$,
	\begin{equation*}
		\Wh(f,\op)^2 \geq d\,\hbar \quad\text{ and }\quad \MKh(\op,\op_2)^2 \geq 2\,d\,\hbar.
	\end{equation*}
	However, some inequalities allow us to compare them to true distances up to terms that vanish when $\hbar$ is small. To present them, let us first introduce the usual correspondence between operators and functions of the phase space. To any sufficiently nice function $f$ of the phase space, one can associate an operator called its Weyl quantization by the formula
	\begin{equation*}
		\op_f := \intdd \widehat{f}(y,\xi)\,e^{2i\pi\,z\cdot\opz} \d y\d w,
	\end{equation*}
	with $z = (y,\xi)$ and $\opz = (x,\opp)$, so that $z\cdot\opz$ is the operator $y\cdot x + \xi\cdot\opp$ where $x$ identified with the operator of multiplication by $x$. Equivalently, this is the operator with integral kernel 
	\begin{equation*}
		\op_f(x,y) = \intd e^{-2i\pi\(y-x\)\cdot\xi} \, f(\tfrac{x+y}{2},h\xi)\d\xi.
	\end{equation*}
	The inverse operation is called the Wigner transform and can be defined as follows. To any operator $\op$ one associates the function
	\begin{equation*}
		f_{\op}(x,\xi) = \intd e^{-i\,y\cdot\xi/\hbar} \,\op(x+\tfrac{y}{2},x-\tfrac{y}{2})\d y
	\end{equation*}
	which satisfies $f_{\op_f} = f$. These two formulas can more generally be understood in the sense of distributions. The Wigner transform satisfies $\intdd f_{\op} = \hd\Tr{\op}$, and $\intdd f_{\op_1}\,f_{\op_2} = \hd\Tr{\op_1 \, \op_2}$, and the transform of a positive operator is always real but might not be nonnegative. Hence, to work with probability distributions, one can look instead at the Husimi transform, which can be defined by
	\begin{equation}\label{eq:Husimi}
		\tilde{f}_{\op} = g_h * f_{\op} \quad \text{ where } \quad g_h(z) = \(2/h\)^d e^{-\n{z}^2/\hbar}.
	\end{equation}
	Similarly, the Weyl quantization of a nonnegative function is a self-adjoint operator but it might not be a positive operator, and so one can consider the Wick quantization or Toeplitz operator associated to $f$, which is the operator
	\begin{equation}\label{eq:Toeplitz}
		\tildop_f := \frac{1}{\hd} \intdd f(z)  \ket{\psi_{z}}\bra{\psi_{z}} \d z
	\end{equation}
	where $\psi_{z}(y) = \(\pi\,\hbar\)^{-\frac{d}{4}}e^{-\n{y-x}^2/(2\hbar)}\, e^{i\,\xi\cdot\(y-\frac{x}{2}\)/\hbar}$ is such that $\op_{g_h} = h^{-d}\, \ket{\psi_0}\bra{\psi_0}$.
	
	Thanks to the Husimi transform, the above pseudometrics can be bounded from below by the classical Wasserstein distance up to a term of order $\hbar$. More precisely, if $f\in\P(\Rdd)$ is a phase space density and $\op,\op_2\in \cP(\h)$ are density operators, the following bounds proved in~\cite[Theorem~2.3]{golse_mean_2016} and \cite[Theorem~2.4]{golse_schrodinger_2017} hold
	\begin{align}\label{eq:W2_Husimi_vs_Wh}
		\W_2(f,\tilde{f}_{\op})^2 &\leq \Wh(f,\op)^2 + d\,\hbar
		\\\label{eq:W2_Husimi_vs_MKh}
		\W_2(\tilde{f}_{\op_1},\tilde{f}_{\op_2})^2 &\leq \MKh(\op_1,\op_2)^2 + 2\,d\,\hbar.
	\end{align}
	One of the main result of this paper is the other direction, that is upper bounds for the pseudometrics. Let us first indicate that upper bounds are known in the particular case of operators that can be written as Toeplitz operators, that is $\op = \tildop_g$ for some probability density $g\in\P(\Rdd)$. Then, as proved again in~\cite[Theorem~2.3]{golse_mean_2016} and \cite[Theorem~2.4]{golse_schrodinger_2017}, it holds
	\begin{align*}
		\Wh(f,\tildop_g)^2 &\leq \W_2(f,g)^2 + d\,\hbar
		\\
		\MKh(\tildop_f,\tildop_g)^2 &\leq \W_2(f,g)^2 + 2\,d\,\hbar.
	\end{align*}
	Notice that it follows, in particular, from the above inequalities that we can deduce the "self-distances" of these operators
	\begin{equation*}
		\Wh(f,\tildop_f)^2 = d\,\hbar \quad \text{ and } \quad \MKh(\tildop_f,\tildop_f)^2 = 2\,d\,\hbar.
	\end{equation*}
	In this case it is even better known (see \cite[Proof of Theorem~2.3]{golse_mean_2016} and \cite[Corollary~3.4]{golse_semiclassical_2021}) that the optimal couplings are respectively
	\begin{equation*}
		\opgam(z) = \frac{1}{\hd}\,f(z)\,\ket{\psi_z}\bra{\psi_z} \quad \text{ and } \quad \opgam(z) = \frac{1}{h^{2d}}\intdd f(z) \ket{\psi_z}\bra{\psi_z}\otimes \ket{\psi_z}\bra{\psi_z} \d z.
	\end{equation*}
	In \cite{golse_semiclassical_2021}, an upper bound was found for more general states, thanks to two main ingredients: the analogue of the triangle inequalities and an upper bound for the "self-distance". The main triangle inequalities for the pseudometrics are the following.
	\begin{prop}[Theorem~3.5, \cite{golse_semiclassical_2021}, Theorem~A, \cite{golse_optimal_2022}]
		If $(f,g,\op)\in\P_2(\Rdd)^2\times\cP_2(\h)$, then
		\begin{equation}\label{eq:triangle_Wh}
			\Wh(f,\op) \leq \W_2(f,g) + \Wh(g,\op).
		\end{equation}
	\end{prop}
	\begin{prop}[Theorem~A, \cite{golse_optimal_2022}]
		If $(g,\op,\op_2)\in\P_2(\Rdd)\times\cP_2(\h)^2$, then
		\begin{equation}\label{eq:triangle_MKh}
			\Wh(\op,\op_2) \leq \Wh(\op,g) + \Wh(g,\op_2).
		\end{equation}
	\end{prop}
	From these and the previous inequalities one can then deduce approximate inequalities in the other cases, such as the case of three operators, see again \cite[Theorem~A]{golse_optimal_2022}. In particular, from the first of these propositions, one deduces that
	\begin{equation*}
		\Wh(f,\op) \leq \W_2(f,\tilde{f}_{\op}) + \Wh(\tilde{f}_{\op},\op)
	\end{equation*}
	and the first term on the right-hand side is a true distance, but the size of the last term, which can also be thought of as a "self-distance", is unclear. A bound on this term was found in \cite{golse_semiclassical_2021} by a weighted norm with $7$ derivatives for the Wigner transform of the square root of the operator $\op$. In this paper, we improve this bound by obtaining an estimate in terms of a norm involving only one derivative of this Wigner transform.

\subsection{The self-distance bound}\label{sec:self-distance}

	One of our main results is the following, which controls the "self-distances" of an operator in the optimal transport pseudometrics.
	\begin{thm}\label{thm:Wh_bound}
		Let $\op\in\cP_2(\h)$. Then
		\begin{align}\label{eq:Wh_bound}
			\Wh(\tilde{f}_{\op},\op)^2 &\leq d\,\hbar + \hbar^2 \,\sNrm{\nabla f_{\sqrt{\op}}}{L^2(\Rdd)}^2
			\\\label{eq:MKh_bound_0}
			\MKh(\op,\op)^2 &\leq 4\,d\,\hbar + 4\,\hbar^2 \,\sNrm{\nabla f_{\sqrt{\op}}}{L^2(\Rdd)}^2.
		\end{align}
	\end{thm}

	\begin{remark}
		In particular, for any operator that is smooth in the sense that $f_{\sqrt{\op}}\in H^1$ uniformly in $\hbar$, there exists a constant $C_{\op}>0$ independent of $\hbar\in(0,1)$ such that $\MKh(\op,\op) \leq C_{\op}\sqrt{\hbar}$ and $\Wh(\tilde{f}_{\op},\op) \leq C_{\op}\sqrt{\hbar}$. As discussed in the previous section, it follows from these results and inequalities~\eqref{eq:W2_Husimi_vs_Wh} and~\eqref{eq:triangle_MKh} that the semiclassical optimal transport is comparable to the classical optimal transport in the sense that
		\begin{equation*}
			\W_2(f,\tilde{f}_{\op}) - \sqrt{d\,\hbar} \leq \Wh(f,\op) \leq \W_2(f,\tilde{f}_{\op}) + C_{\op}\sqrt{\hbar}.
		\end{equation*}
	\end{remark}
	
	One might notice that while Inequality~\eqref{eq:Wh_bound} gives an upper bound of order $\hbar^2$ to the positive quantity $\Wh(\tilde{f}_{\op},\op)^2-d\,\hbar$ in the case when $f_{\sqrt{\op}}\in H^1$ uniformly in $\hbar$, this is not the case of Inequality~\eqref{eq:MKh_bound_0} which only gives an upped bound of order $\hbar$ to the positive quantity $\MKh(\op,\op)^2-2\,d\,\hbar$. Interestingly, one can get a better order in $\hbar$ by considering instead the distance between $\op$ and its "semiclassical convolution by a Gaussian" $\ttildop = \tildop_{\tilde{f}_{\op}}$, thanks to the following inequality proved in Proposition~\ref{prop:MKh_vs_Wh},
	\begin{equation}\label{eq:MKh_vs_Wh}
		\MKh(\tildop_f,\op)^2 \leq \Wh(f,\op)^2 + d\, \hbar,
	\end{equation}
	which allows to compare the quantum and semiclassical pseudometrics. By~\cite[Theorem~5.4]{golse_optimal_2022}, a similar inequality holds in the other direction in the form
	\begin{equation}\label{eq:Wh_vs_MKh}
		\Wh(\tilde{f}_{\op_1},\op_2)^2 \leq \MKh(\op_1,\op_2)^2 + d\, \hbar.
	\end{equation}
	It gives the following result.
	\begin{thm}\label{thm:MKh_bound}
		Let $\op\in\cP_2(\h)$ and define $\ttildop := \tildop_{\tilde{f}_{\op}}$. Then it holds
		\begin{equation}\label{eq:MKh_bound}
			\MKh(\op,\ttildop)^2 \leq 2\,d\,\hbar + \hbar^2 \,\sNrm{\nabla f_{\sqrt{\op}}}{L^2(\Rdd)}^2.
		\end{equation}
	\end{thm}

	Let us recall that the $L^2$ norm of the Wigner transform is proportional to the Hilbert--Schmidt norm of its associated operator, and more precisely, one can introduce semiclassical Schatten norms 
	\begin{equation}\label{eq:def_norm}
		\Nrm{\op}{\L^p} := h^\frac{d}{p} \Nrm{\op}{p} = h^\frac{d}{p} \Tr{\n{\op}^p}^\frac{1}{p}
	\end{equation}
	so that $\Nrm{\op}{\L^2} = \Nrm{f_{\op}}{L^2(\Rdd)}$. Another remark is the fact that the gradient of the Wigner transform corresponds to Wigner transforms of commutators, as given by the correspondence principle. One can introduce quantum gradients
	\begin{equation}\label{eq:quantum_gradients}
		\Dhx \op := \com{\nabla,\op} \quad \text{ and } \quad
		\Dhv \op := \com{\frac{x}{i\hbar},\op},
	\end{equation}
	so that $f_{\Dhx \op} = \Dx f_{\op}$ and $f_{\Dhv \op} = \Dv f_{\op}$. More generally, we will write $\Dh\op := (\Dhx\op,\Dhv\op)$, so that $\n{\Dh\op}^2 = \n{\Dhx\op}^2 + \n{\Dhv\op}^2$. Hence, one can for example rewrite Inequality~\eqref{eq:Wh_bound} in terms of quantum Sobolev norms as follows
	\begin{align*}
		\Wh(\tilde{f}_{\op},\op)^2 &\leq d\,\hbar + \hbar^2 \Nrm{\Dh{\sqrt{\op}}}{\L^2}^2.
	\end{align*}
	This can also be written
	\begin{equation*}
		\Wh(\tilde{f}_{\op},\op)^2 \leq d\,\hbar + \Nrm{\com{x,\sqrt{\op}}}{\L^2}^2 + \Nrm{\com{\hbar\nabla,\sqrt{\op}}}{\L^2}^2.
	\end{equation*}
	The Hilbert--Schmidt norms of the commutator with the square root of an operator is known as the Wigner--Yanase skew information. More precisely, if we define
	\begin{equation*}
		\cI_K(\op) := \frac{1}{2} \Nrm{\com{K,\sqrt{\op}}}{\L^2}^2
	\end{equation*}
	then Inequality~\eqref{eq:Wh_bound} can also be written 
	\begin{equation*}
		\Wh(\tilde{f}_{\op},\op)^2 \leq d\,\hbar + 2\,\cI_x(\op) + 2\,\cI_{\opp}(\op).
	\end{equation*}
	In the particular case of pure states of the form $\op = h^{-d}\ket{\psi}\bra{\psi}$ with $\psi$ of norm $1$, the skew information is nothing but the variance of $K$ for the state $\psi$
	\begin{equation*}
		\cI_K(\op) = \frac{1}{2}\Tr{\n{\com{K,\ket{\psi}\!\bra{\psi}}}^2} = \bangle{|K-\bangle{K}_\psi|^2}_\psi =: \sigma_K(\psi)^2
	\end{equation*}
	where $\bangle{K}_\psi = \Inprod{\psi}{K\psi}$. In particular, Inequality~\eqref{eq:Wh_bound} is a generalization of \cite[Proposition~9.1]{golse_semiclassical_2021} which states that
	\begin{equation}\label{eq:golse_paul}
		\Wh(\tilde{f}_{\op},\op)^2 \leq 2\,\sigma_x(\psi)^2 + 2\,\sigma_{\opp}(\psi)^2 + 2\,d\,\hbar.
	\end{equation}
	Another generalization of the above formula to the case of density operators is given in~\cite{golse_semiclassical_2021} by replacing $\bangle{A}_{\psi}$ by $\bangle{A}_{\op} = \hd\Tr{\sqrt{\op}\,A\,\sqrt{\op}}$. This gives for example
	\begin{equation*}
		\sigma_x(\op)^2 = \hd\Tr{\sqrt{\op} \,|x-\bangle{x}_{\op}\!|^2\sqrt{\op}} = \bangle{\n{x}^2}_{\op} -\bangle{x}_{\op}^2.
	\end{equation*}
	The quantity on the right-hand side of Formula~\eqref{eq:golse_paul} will only converge to $0$ in the semiclassical limit $\hbar\to 0$ in the case of concentrating solutions, that is if the Wigner transform $f_{\op}$ converges to a Dirac delta in phase space. In general, by the generalized Heisenberg inequality, $\sigma_x(\op)^2 + \sigma_{\opp}(\op)^2 \geq 2\, \sigma_x(\op)\,\sigma_{\opp}(\op) \geq d\,\hbar$. On the contrary, Inequality~\eqref{eq:Wh_bound} applies to a large category of operators. For $K$ self-adjoint, the quantity $\cI_K(\op)$ is indeed smaller than the variance of $K$ as follows by writing $\cI_K(\op) = \hd\Tr{\op\,K^2} - \hd\Tr{\sqrt{\op}K\sqrt{\op}K}$ and $\sigma_K(\op)^2 = \hd\Tr{\op\,K^2} - h^{2d}\Tr{\op\,K}^2$ and noticing that by the Cauchy--Schwarz inequality
	\begin{equation*}
		\Tr{\op\,K}^2 = \Tr{\sqrt{\op}\,\op^\frac{1}{4}\,K\,\op^\frac{1}{4}}^2 \leq \Tr{\op} \Tr{\n{\op^\frac{1}{4}\,K\,\op^\frac{1}{4}}^2} = \frac{\Tr{\sqrt{\op}K\sqrt{\op}K}}{\hd}.
	\end{equation*}
	An analogue of the Heisenberg inequality for the Skew information is the quantum Sobolev inequality proved in~\cite{lafleche_quantum_2022} which states that there exists a constant $C_d$ depending only on the dimension $d$ such that
	\begin{equation*}
		2\,\sqrt{\cI_x(\op)\,\cI_{\opp}(\op)} = \hbar^2 \Nrm{\Dhx{\sqrt{\op}}}{\L^2}\Nrm{\Dhv{\sqrt{\op}}}{\L^2} \geq C_d\,\hbar^2\Nrm{\op}{\L^\frac{d}{d-1}}.
	\end{equation*}
	The power $\hbar^2$ appearing on the left-hand side of the above formula is sharp, showing that the skew information can indeed be much smaller than the variance in general.
	
	To conclude this section, let us compare our theorem with an upper bound obtained in~\cite{golse_semiclassical_2021}. Let $w := (1+ \n{x}^2+\n{\xi}^2)^{-n-1/2}$ with $n>d$ and $\opm := \op_w^{-1/2}$. Then by the proof of~\cite[Theorem~2.3]{golse_semiclassical_2021}
	\begin{equation}\label{eq:Wh_weighted_Linfty}
		\Wh(\tilde{f}_{\op},\op)^2 \leq C\, \sup_{O\in\Omega} \Nrm{\opm\sqrt{\op}\com{O,\sqrt{\op}\,\opm}}{\L^\infty}
	\end{equation}
	where $O = \set{\n{x}^2+\n{\opp}^2, x_j, \opp_j \text{ for } j\in\set{1,\dots,d}}$. It follows from the definitions~\eqref{eq:def_norm} and~\eqref{eq:quantum_gradients} above that $\com{x_j,\opmu} = i\hbar\,\Dhvj{\opmu}$, $\com{\opp_j,\opmu} = -i\hbar\,\Dhxj{\opmu}$ and by the Leibniz formula for commutators,
	\begin{align*}
		\com{\n{x}^2,\opmu} &= x\cdot\com{x,\opmu} + \com{x,\opmu}\cdot x = i\hbar\(x\cdot\Dhv{\opmu} + \Dhv{\opmu}\cdot x\)
		\\
		\com{\n{\opp}^2,\opmu} &= \opp\cdot\com{\opp,\opmu} + \com{\opp,\opmu}\cdot \opp = i\hbar\(\opp\cdot\Dhx{\opmu} + \Dhx{\opmu}\cdot \opp\)
	\end{align*}
	and so Equation~\eqref{eq:Wh_weighted_Linfty} can be interpreted as an estimate by a weighted analogue of the $W^{1,\infty}$ norm. It implies for example
	\begin{equation*}
		\Wh(\tilde{f}_{\op},\op)^2 \leq C\,\hbar\,  \Nrm{\opm\sqrt{\op}}{\L^\infty} \Nrm{\bangle{\opz}\Dh(\sqrt{\op}\,\opm)}{\L^\infty}
	\end{equation*}
	with $\bangle{\opz}^2 = 1+\n{x}^2+\n{\opp}^2$. By \cite{lafleche_optimal_2023}, this kind of bound is however not compatible with projection operators, at the contrary of Theorem~\ref{thm:Wh_bound}, as will be proved in Section~\ref{sec:projections}. Theorem~\ref{thm:Wh_bound} also improves the above bound as it gives a correction of order $\hbar^2$ to the positive quantity $\Wh(\tilde{f}_{\op},\op)^2-d\,\hbar$ in the case of smooth operators.
	
	\subsection{Application: thermal states, spectral projections and Toeplitz operators.} In this section, estimates on the "self-distance" for several particular families of operators are obtained.
	
	\subsubsection{Thermal states} Since $\Wh(\tilde{f}_{\op},\op)^2 \geq d\,\hbar$, Inequality~\eqref{eq:Wh_bound} has sharp first order term in $\hbar$ for nice operators such that $\|\Dh{\sqrt{\op}}\|_{\L^2}$ is bounded uniformly in $\hbar$. A physically important family of operators for which this is the case are the thermal states. They are states of the form
	\begin{equation}\label{eq:thermal}
		\op = Z_\beta^{-1}\,e^{-\beta \,H}
	\end{equation}
	with an Hamiltonian given by
	\begin{equation}
		H = -\hbar^2\Delta + V
	\end{equation}
	where the potential $V$ is the operator of multiplication by the function $x\mapsto V(x)$. The normalization constant $Z_\beta$ is the usual partition function $Z_\beta = \hd\Tr{e^{-\beta \,H}}$. We will assume that $V$ converges to $+\infty$ at infinity, and more precisely that there exist some positive constants $a$, $b$, $\kappa$ and $\kappa_2$ such that for any $x\in\Rd$,
	\begin{align}\label{eq:V_coercivity}
		V(x) &\geq \kappa\n{x}^a
		\\\label{eq:V_regu}
		\n{\nabla V(x)} &\leq \kappa_2 \n{x}^b.
	\end{align}
	It gives the following result.
	\begin{prop}\label{prop:thermal}
		Let $\op$ be of the form~\eqref{eq:thermal} with $V$ verifying \eqref{eq:V_coercivity} and \eqref{eq:V_regu}. Then it holds
		\begin{equation*}
			\Wh(\tilde{f}_{\op},\op)^2 \leq d\,\hbar + C\,\hbar^2
		\end{equation*}
		where $C \leq C_{d,a,b}\, Z_\beta^{-1} \,\kappa^{-d/a}\,\beta^{1-d\(\frac{1}{a}+\frac{1}{2}\)} \(1 + \kappa_2\,\beta\, \max\((\kappa\hbar)^{2b/(1+a)}, (\kappa\beta)^{-b/a}\)\)$ for some constant $C_{d,a,b}$ depending only on $d$, $a$ and $b$. In particular, $C$ is uniformly bounded in $\hbar\in (0,1)$.
	\end{prop}
	
	The semiclassical smoothness of other type of thermal states such as the Fermi--Dirac equilibrium is also considered in~\cite{chong_semiclassical_2023}, which leads to the same bound with another constant $C$.
	
	\subsubsection{Spectral projections}\label{sec:projections} Another important case of application of Inequality~\eqref{eq:Wh_bound} is the case of projection operators such as spectral functions of Schr\"odinger operators. These operators correspond for example to the one-particle density operators associated to Slater determinants. More precisely, we consider operators of the form
	\begin{equation}\label{eq:spectral_proj}
		\op = Z_0^{-1}\,\indic_{(-\infty,0]}(-\hbar^2\Delta + V(x))
	\end{equation}
	where $V$ is such that there exists $\eps > 0$ and open sets $\Omega_\eps$ bounded and $\Omega$ verifying $\overline{\Omega_\eps} \subset \Omega\subset\Rd$ such that
	\begin{system}\label{eq:conditions_u}
		&V\in C^\infty(\Omega) \cap L^1_\loc(\Omega^c)
		\\
		&V \geq \eps \text{ on }\Omega_\eps^c,
	\end{system}
	and $Z_0 = \hd\Tr{\indic_{(-\infty,0]}(-\hbar^2\Delta + V(x))}$. This normalization constant is bounded uniformly in $\hbar$ by the Cwikel--Lieb--Rozenblum inequality~\cite{cwikel_weak_1977}. By the Weyl law, it converges when $\hbar\to 0$ to $\iintd \indic_{\n{\xi}^2 + V(x)\leq 0} \d x\d\xi = \frac{\omega_d}{d} \intd V_-^{d/2}$ where $\omega_d$ is the size of the unit sphere of $\R^d$ and $V_- = \max(0,-V)$ is the negative part of $V$.
	
	As proved in~\cite{lafleche_optimal_2023}, such operators are not as smooth as thermal states as they converge to characteristic functions of the phase space in the limit $\hbar\to 0$. Inequality~\eqref{eq:Wh_bound} is however still sufficiently strong to give the following result.

	\begin{prop}\label{prop:pure_states}
		Assume that $\op$ is of the form~\eqref{eq:spectral_proj} with $V$ satisfying Hypothesis~\eqref{eq:conditions_u}. Then there exists a constant $C>0$ independent of $\hbar$ such that
		\begin{equation*}
			\Wh(\tilde{f}_{\op},\op) \leq C \,\sqrt{\hbar}.
		\end{equation*}
	\end{prop}
	When $V$ is less regular however, it is expected that there might be cases where $\sqrt{\hbar} \Nrm{\Dh{\op}}{\L^2}$ does not converge to $0$ (see \cite{deleporte_universality_2021, lafleche_optimal_2023}). It is not clear however that there are potentials such that the above proposition is false.
	
	\subsubsection{Powers of Toeplitz operators} Another particular case is when the operator can be written as a suitably normalized even power of a Wick (Toeplitz) quantization. In this case, no derivative is needed on the symbol of the operator. Moreover, in the case of additional regularity of the symbol, one would obtain a result similar to the one of thermal states with a sharp constant for the first order term in $\hbar$. This generalizes the case of Toeplitz operators proved in \cite{golse_mean_2016}.
	\begin{prop}\label{prop:toplitz_power}
		Let $f\in L^\infty(\Rdd)\cap L^1(\Rdd)\setminus \set{0}$ and $n$ be an even nonzero integer, and define
		\begin{equation*}
			\op := C_f\,\tildop_f^n,
		\end{equation*}
		where $C_f = \sNrm{\tildop_f}{\L^n}^{-n}$ is such that $\hd \Tr{\op} = 1$. Then, the following inequality holds
		\begin{align}\label{eq:power_toplitz}
			\Wh(\tilde{f}_{\op},\op)^2 &\leq \(d + \tfrac{n^2}{2\,e}\,C_f \Nrm{f}{L^\infty(\Rdd)}^{n-2} \Nrm{f}{L^2(\Rdd)}^2\) \hbar
		\end{align}
		where the normalization factor $C_f$ is bounded uniformly in $\hbar \in (0,1)$ and satisfies 
		\begin{equation}\label{eq:Cf_bound}
			C_f \leq \Nrm{f}{L^1(\Rdd)}^{n-2} \Nrm{g_1*f}{L^2(\Rdd)}^{2\(1-n\)}.
		\end{equation}
	\end{prop}
	
	\begin{remark}
		When $n=2$, then it is sufficient to have $f\in L^2(\Rdd)\setminus \set{0}$, and the above result yields
		\begin{equation*}
			\Wh(\tilde{f}_{\op},\op)^2 < \(d + C_f \Nrm{f}{L^2(\Rdd)}^2\) \,\hbar.
		\end{equation*}
		with $C_f \leq \Nrm{g_1*f}{L^2(\Rdd)}^{-2}$.
	\end{remark}
	
\subsection{Quantum optimal transport and Sobolev spaces}

	In this section, we compare the quantum optimal transport distances with the quantum Sobolev norms. We recall the definition of these quantum norms, as defined in \cite{lafleche_quantum_2022}. Let us first define the phase space translation operators by
	\begin{equation*}
		\sfT_z\op := \tau_z\,\op\,\tau_{-z}
	\end{equation*}
	where $\tau_z$ is the operator defined for $z_0 = (x_0,\xi_0)\in\Rdd$ and for any $\varphi\in L^2(\Rd)$ by $\tau_{z_0}\varphi(x) = e^{-i\,\xi_0\cdot\(x -\frac{x_0}{2}\)/\hbar}\, \varphi(x-x_0)$. They satisfy $\sfT_{z+z'} = \sfT_z\sfT_{z'}$ and $f_{\sfT_{z}\op}(z') = f_{\op}(z'-z)$. Then, for $p\in[1,\infty]$ and $s\in\R$ we define for any compact operator $\op$ its quantum Sobolev norm of order $s$ by
	\begin{align*}
		\Nrm{\op}{\dot{\cW}^{s,p}} &= \Nrm{\op}{\L^p} &&\text{ if } s = 0
		\\
		\Nrm{\op}{\dot{\cW}^{s,p}} &= \Nrm{\Dh\op}{\L^p} &&\text{ if } s = 1
		\\
		\Nrm{\op}{\dot{\cW}^{s,p}} &= c_{s,p} \Nrm{\frac{\Nrm{\sfT_z\op - \op}{\L^p}}{\n{z}^{s+2d/p}}}{L^p(\Rdd)} &&\text{ if } s \in (0,1)
		\\
		\Nrm{\op}{\dot{\cW}^{s,p}} &= \sup_{\Nrm{B}{\dot{\cW}^{-s,p'}} \leq 1} \hd\Tr{\op B} &&\text{ if } s < 0
	\end{align*}
	where $c_{s,p} = 1$ if $p=\infty$ and $c_{s,p}^p = \frac{p\,\n{\omega_{-2s}}}{4\,\omega_{2d+sp}} \(\frac{\pi\,\omega_{p+1}}{s^{(p-2)/2}}\)^s$ with $\omega_d = |\SS^{d-1}| = \frac{2\,\pi^{d/2}}{\Gamma(d/2)}$ in the other cases, and where $p' = \frac{p}{p-1}$ denotes the H\"older conjugate\footnote{With $p'=\infty$ when $p=1$.}. The choice of the constant $c_{s,p}$ implies in particular that
	\begin{equation*}
		\Nrm{\op}{\dot{\cH}^s} := \Nrm{\op}{\dot{\cW}^{s,2}} = \Nrm{f_{\op}}{\dot{H^s}(\Rdd)} = \sNrm{(-\Delta)^{s/2}f_{\op}}{L^2(\Rdd)}.
	\end{equation*}
	
	We start by a proposition giving an upper bound to the quantum optimal transport cost in terms of quantum Sobolev norms with $p=2$. This is a quantitative analogue of the fact that bounded moments and weak convergence of measures implies convergence in Wasserstein distances (see e.g.~\cite[Theorem~5.11]{santambrogio_optimal_2015}).
	\begin{thm}\label{thm:upper_bound}
		Let $n> 2$ and $(f,\op,\op_2)\in\P_n(\Rdd)\times\cP_n(\h)^2$ uniformly in $\hbar$. Then there exists constants $C = C_{n,f,\op}$ and $C_2 = C_{n,\op,\op_2}$ independent of $\hbar\in(0,1)$ such that
		\begin{align*}
			\Wh(f,\op) &\leq C \Nrm{f - \tilde{f}_{\op}}{\dot{W}^{-1,1}(\Rdd)}^\theta + \sqrt{d\hbar} + D_{\op} \,\hbar
			\\
			\MKh(\op,\op_2) &\leq C_2 \Nrm{\op - \op_2}{\dot{\cW}^{-1,1}}^\theta + 2\,\sqrt{d\hbar} + \(D_{\op_1}+D_{\op_2}\) \hbar,
		\end{align*}
		where $D_{\op} = \|\Dh{\sqrt{\op}}\|_{\L^2}$ and $\theta = \frac{n-2}{2\(n-1\)}$.
	\end{thm}

	The quantity $D_{\op}$ is the quantity discussed in Section~\ref{sec:self-distance}, and so is for example uniformly bounded with respect to $\hbar$ for operators with uniformly-in-$\hbar$ smooth symbols. The difference $\sNrm{f - \tilde{f}_{\op}}{\dot{W}^{-1,1}(\Rdd)}$ can be bounded both by a quantum or a classical Sobolev norm. More precisely, it holds
	\begin{align*}
		\Nrm{f - \tilde{f}_{\op}}{\dot{W}^{-1,1}(\Rdd)} &\leq \Nrm{f - f_{\op}}{\dot{W}^{-1,1}(\Rdd)} + c_d \sqrt{\hbar}
		\\
		\Nrm{f - \tilde{f}_{\op}}{\dot{W}^{-1,1}(\Rdd)} &\leq \Nrm{\op - \op_f}{\dot{\cW}^{-1,1}(\Rdd)} + c_d\sqrt{\hbar},
	\end{align*}
	where $c_d = \frac{\Gamma(d+1/2)}{\Gamma(d)} < \sqrt{d}$, as follows by using the fact that $\sNrm{f - g_h*f}{\dot{W}^{-1,1}(\Rdd)} \leq c_d\sqrt{\hbar} \Nrm{f}{L^1}$ and the triangle inequality.
	
	In the other direction, it is also possible to bound by below the quantum optimal transport cost by negative order Sobolev norms, as was already proved in~\cite{golse_convergence_2021, golse_semiclassical_2021}. More precisely, it is proved in~\cite[Proposition~B.5]{golse_convergence_2021} that
	\begin{equation*}
		\hd \Nrm{\op-\op_2}{\cW^{-2-\floor{d/2},\infty}} \leq \MKh(\op,\op_2) + C_d\,\sqrt{\hbar}
	\end{equation*}
	for some constant $C_d$ depending only on the dimension. The left-hand side of the inequality can however be not very well behaved in terms of the Planck constant because of the $\hd$ appearing in front. In~\cite[Proposition~B.3]{golse_semiclassical_2021}, it is proved that
	\begin{equation*}
		\sNrm{f_{\op}-f_{\op_2}}{(\mathrm{Lip}\cap H^1)'} \leq \MKh(\op,\op_2) + \sqrt{\hbar} \Nrm{\op-\op_2}{\L^2}.
	\end{equation*}
	where $\Nrm{f}{(\mathrm{Lip}\cap H^1)'} = \sup \intdd f\,\varphi$ with the supremum taken over all $\varphi$ such that $\Nrm{\varphi}{W^{1,\infty}} \leq 1$ and $\Nrm{\varphi}{H^1} \leq 1$. As noticed in \cite[Proposition~B.2]{golse_semiclassical_2021}, the norm appearing on the left-hand side of the inequality can be compared to quantum Sobolev norms of operators by the inequality
	\begin{equation*}
		(h/2)^d \Nrm{\op-\op_2}{\cW^{-1,\infty}} \leq \sNrm{f_{\op}-f_{\op_2}}{(\mathrm{Lip}\cap H^1)'}\,.
	\end{equation*}
	However, this makes again appear a positive power of the Planck constant on the left-hand side. We propose here as another alternative the following quantum analogue of the classical Loeper inequality~\cite{loeper_uniqueness_2006}.
	\begin{thm}\label{thm:lower_bound}
		Let $d>1$ and $\(f,\op,\op_2\)\in\P(\Rdd)\otimes\cP(\h)^2$. Then for any $s\in[-1,1]$,
		\begin{align*}
			\Nrm{\op-\op_f}{\dot{\cH}^{-1}} &\leq C_\infty^{1/2} \(\Wh(f,\op)^2 + d\,\hbar\)^{1/2} + \(\tfrac{\hbar}{4}\)^\frac{s+1}{2} \Nrm{\op}{\dot{\cH}^s}
			\\
			\Nrm{\op-\op_2}{\dot{\cH}^{-1}} &\leq \cC_\infty^{1/2} \(\MKh(\op,\op_2)^2 + 2\,d\,\hbar\)^{1/2} + \(\tfrac{\hbar}{4}\)^\frac{s+1}{2} \Nrm{\op-\op_2}{\dot{\cH}^s}
		\end{align*}
		where $C_\infty = \max\!\(\Nrm{f}{L^\infty(\Rdd)}, \Nrm{\op}{\L^\infty}\)$ in the first equation, and where $\cC_\infty = \max\!\(\Nrm{\op}{\L^\infty}, \Nrm{\op_2}{\L^\infty}\)$ in the second equation.
	\end{thm}
	Notice that the left-hand side of the above inequalities could also be written in terms of the classical $\dot{H}^{-1}(\Rdd)$ norm of the Wigner transform since $\Nrm{f_{\op}}{L^2(\Rdd)} = \Nrm{\op}{\L^2}$, so that, for example,  $\sNrm{\op-\op_f}{\dot{\cH}^{-1}} = \Nrm{f_{\op}-f}{\dot{H}^{-1}(\Rdd)}$. These inequalities are of course particularly interesting when $\cC_\infty$ is bounded uniformly with respect to $\hbar$, which is the case for most of the previously discussed examples and is often relevant in the case of fermions due to the Pauli principle (see e.g.~\cite{chong_many-body_2021}). The remainder term on the right-hand side is smaller when $\op-\op_2 \in \cH^s$ with $s$ larger. In the particular case when $s=0$, one obtains the following simplified inequality.
	
	\begin{cor}
		Let $d>1$ and $\(\op,\op_2\)\in\cP(\h)^2$ satisfy $0\leq \op,\op_2\leq 1$. Then it holds
		\begin{align*}
			\Nrm{f-f_{\op}}{\dot{H}^{-1}(\Rdd)} &\leq  \Wh(f,\op) + \(\tfrac{1}{2}+\sqrt{d}\)\sqrt{\hbar}
			\\
			\Nrm{\op-\op_2}{\dot{\cH}^{-1}} &\leq \MKh(\op,\op_2) + \(1+\sqrt{2d}\) \sqrt{\hbar}.
		\end{align*}
	\end{cor}

\subsection{Application to the mean-field and semiclassical limits}

	One of the main applications of the quantum optimal transport pseudometrics described above is that they allow to get quantitative bounds on the errors made when using approximate equations in the case of a large number of particles or in the context of a semiclassical approximation~\cite{golse_mean_2016, golse_schrodinger_2017, lafleche_propagation_2019, lafleche_global_2021, porat_magnetic_2022}. This is in particular useful in the context of numerical schemes \cite{golse_convergence_2021}. As an example of application of the above bounds, we consider the Hartree equation
	\begin{equation}\label{eq:Hartree}
		i\hbar\,\dpt\op = \com{H_{\op},\op}
	\end{equation}
	with Hamiltonian $H_{\op} = \frac{-\hbar^2\Delta}{2} + V_{\op}(x)$ where the mean-field potential $V_{\op}$ is defined by $V_{\op}(x) = K*\rho(x) = \intd K(x-y)\,\rho(y)\d y$ with $\rho(x) = \hd\,\op(x,x)$, and the Vlasov equation
	\begin{equation}\label{eq:Vlasov}
		\dpt f + v\cdot\Dx f + E\cdot\Dx f = 0
	\end{equation}
	where $E = -\nabla V_f = -\nabla K*\rho_f$ with $\rho_f(x) = \intd f(x,\xi)\d\xi$. Then the following result follows from the above bounds and~\cite[Theorem~2.5]{golse_schrodinger_2017}.
	\begin{thm}
		Let  $K$ be an even bounded function of class $C^{1,1}(\Rd)$. Then for any solution $f$ of the Vlasov equation with initial data $f^\init\in\P_2(\Rdd)$ and for any solution $\op$ of the Hartree equation with initial data $\op^\init\in\cP_2(\h)$, it holds
		\begin{equation*}
			W_2(f,\tilde{f}_{\op})^2 \leq  e^{\lambda\,t}\(W_2(f^\init,\tilde{f}_{\op^\init})^2 + 2\,d\,\hbar + D_{\op}^2 \,\hbar^2\)
		\end{equation*}
		where $D_{\op} = \|\Dh{\sqrt{\op^\init}}\|_{\L^2}$ and $\lambda = 1+\max(1,4\Nrm{\nabla K}{W^{1,\infty}}^2)$.
	\end{thm}
	
	When $f^\init\in L^\infty(\Rdd)$ and $\op^\init\in\L^\infty$ both have moments higher than $2$ bounded, then one can express the above result in terms of negative Sobolev norms and of the Wigner transform.
	
	\begin{thm}
		Let $n>2$ and $K$ be an even bounded function of class $C^{1,1}(\Rd)$. Then for any solution $f$ of the Vlasov equation with initial data $f^\init\in\P_n(\Rdd)$ satisfying $f^\init\leq 1$ and for any solution $\op$ of the Hartree equation with initial data $\op^\init\in\cP_n(\h)$ uniformly in $\hbar$ satisfying $\op^\init\leq 1$, there exists a constant $C$ independent of $\hbar$ such that
		\begin{equation*}
			\Nrm{f-f_{\op}}{\dot{H}^{-1}(\Rdd)} \leq  e^{\lambda\,t}\(C \Nrm{f^\init - \tilde{f}_{\op^\init}}{\dot{W}^{-1,1}(\Rdd)}^\theta + 3\sqrt{d\hbar} + D_{\op} \,\hbar\)
		\end{equation*}
		where $D_{\op} = \|\Dh{\sqrt{\op^\init}}\|_{\L^2}$, $\theta = \frac{n-2}{2\(n-1\)}$, $\lambda = 1+2 \Nrm{\nabla K}{W^{1,\infty}}^2$, and $C$ depends only on $n$, $d$ and the moments of the densities at initial time.
	\end{thm}
	One can similarly modify the results of \cite{golse_mean_2016, golse_schrodinger_2017, lafleche_propagation_2019, lafleche_global_2021, porat_magnetic_2022} to get results with more singular potentials, magnetic fields and results for the mean-field limit.

\section{Proof of the "self-distance" bounds}

\subsection{Classical case}

	Let us first comment on the strategy by giving a simple proof that can be thought of as the analogous proof in the classical setting. The idea is to replace $\Wh(\tilde{f}_{\op},\op)$ by the classical quantity $W_2(g_h*f,f)$ with $g_h$ given by Equation~\eqref{eq:Husimi} and $f\in\P(\Rdd)$ some probability density on the phase space. It is indeed not difficult to find an upper bound to this latter quantity. Indeed, for any $z\in\Rdd$,
	\begin{equation*}
		\W_2(g_h(\cdot-z), \delta_z)^2 = d\,\hbar
	\end{equation*}
	with optimal transport plan $\gamma_z(z_1,z_2) = g_h(z_1-z)\,\delta_z(z_2)$. This is a measure of how close a Gaussian function and a Dirac measure are. More generally, since $f(z) = \intdd f \,\delta_z$ can be seen as a convex combination of Dirac measures and $g_h * f = \intdd f\,g_h(\cdot-z)$ can be seen as a convex combination of Gaussian functions, we can look at the coupling obtained by convex combination under the form
	\begin{equation*}
		\gamma(z_1,z_2) = \intdd f(z)\, \gamma_z \d z = \intdd f(z)\, g_h(z_1-z)\,\delta_z(z_2) \d z = f(z_1)\,g_h(z_1-z_2).
	\end{equation*}
	This is indeed a coupling of $f$ and $g_h * f$. It gives
	\begin{equation*}
		\int_{\R^{4d}} \n{z_1-z_2}^2 \gamma(\d z_1,\d z_2) = \intdd f * (\n{z}^2g_h) = \intdd \n{z}^2g_h(z)\d z = d\hbar,
	\end{equation*}
	from which we deduce that
	\begin{equation*}
		W_2(g_h * f, f)^2 \leq d\,\hbar.
	\end{equation*}
	The proof in the quantum case will be a non-commutative analogue of the above proof. It requires the control of additional error terms due to the non-commutativity.
	
\subsection{Quantum case}

	In this section, we prove the first inequality of Theorem~\ref{thm:Wh_bound}, which is the main novelty of this paper and follows immediately from the next proposition.

	\begin{prop}\label{prop:Wh_bound_0}
		Let $\op\in\cP_2(\h)$. Then
		\begin{equation}\label{eq:Wh_bound_0}
			\Wh(\tilde{f}_{\op},\op) \leq \sqrt{d\hbar + \hbar^2 \Nrm{\Dh{\sqrt{\op}}}{\L^2}^2}.
		\end{equation}
	\end{prop}
	
	\begin{remark}
		More generally, one could replace $\tilde{f}_{\op}$ by any function of the form $f = g * f_{\op}$ where $g$ is a (possibly $\hbar$ dependent) non-negative function which is the Weyl quantization of a positive operator, and is such that $\intdd z \,g(z)\d z = 0$. Then $d\hbar$ has to be replaced by $\intdd \n{z}^2 g(z)\d z$.
	\end{remark}
	
	\begin{remark}
		Notice that for any $s\in[0,1]$,
		% (a+b)^(2-2s) < 2^{(1-2s)_+} (a^(2-2s)+b^(2-2s))
		\begin{multline*}
			\Nrm{\Dhv\opmu}{\L^2}^2 = \frac{\hd}{\hbar^2}\intdd \n{x-y}^2 \n{\op(x,y)}^2\d x\d y
			\\
			\leq \frac{C_s\,\hd}{\hbar^{2\(1-s\)}}\intdd \n{x}^{2(1-s)} \n{\DDh_{\xi}^{s/2}\opmu(x,y)}^2\d x\d y
		\end{multline*}
		where $C_s = 2^{1+(1-2s)_+}$, and similarly by exchanging $x$ and $\xi$. Hence
		\begin{equation*}
			\hbar^2 \Nrm{\Dh{\sqrt{\op}}}{\L^2}^2 \leq C_s\,\hbar^{2s}\(\Nrm{\n{x}^{1-s}\DDh_{\xi}^{s/2}\sqrt{\op}}{\L^2}^2 + \Nrm{\n{\opp}^{1-s}\DDh_x^{s/2}\sqrt{\op}}{\L^2}^2\)
		\end{equation*}
		Note that $\DDh_x$ commutes with $\n{\opp}$ and $\DDh_\xi$ commutes with $\n{x}$. Since $\Delta_x f_{\op} = f_{\DDh_x\op}$ this can be written in terms of the Wigner transform as
		\begin{equation*}
			\hbar^2 \Nrm{\Dh{\sqrt{\op}}}{\L^2}^2 \leq C_s\,\hbar^{2s}\(\Nrm{f_{\n{x}^{1-s}\sqrt{\op}}}{H^s}^2 + \Nrm{f_{\n{\opp}^{1-s}\sqrt{\op}}}{H^s}^2\).
		\end{equation*}
	\end{remark}
	
	\begin{remark}
		In the classical case, the following inequality holds $\sNrm{\nabla \sqrt{f}}{L^2} \leq \Nrm{\nabla^2 f}{L^1}$. One can conjecture that the same is true in the non-commutative setting, that is for any compact operator $\op$, it holds $\sNrm{\Dh{\sqrt{\op}}}{\L^2}^2 \leq \sNrm{\Dh^2\op}{\L^1}$.
	\end{remark}
	
	We start by recalling some basic facts about operators and trace. First if $B\geq 0$ is a bounded operator and $A\geq 0$ is a densely defined operator, then
	\begin{equation*}
		\Tr{\sqrt{B}\,A\,\sqrt{B}} = \Nrm{\sqrt{A}\sqrt{B}}{2}^2 = \Nrm{\sqrt{B}\sqrt{A}}{2}^2 = \Tr{\sqrt{A}\,B\,\sqrt{A}}
	\end{equation*}
	since the singular values are unchanged by taking the adjoint (see e.g. \cite[Equation~1.3]{simon_trace_2005}). In the particular case when $AB$ is trace class, then the above quantity is also equal to $\Tr{AB}$. We will also use several times the following lemma.
	
	\begin{lem}\label{lem:integral_convolution}
		Let $\op$ and $\opmu$ be two trace class operators such that $\(1+\n{x}^n+\n{\opp}^n\) \op$ and $\(1+\n{x}^n+\n{\opp}^n\) \opmu$ are Hilbert--Schmidt operators for some $n>d$. Then it holds
		\begin{equation*}
			\intdd \hd \Tr{\op\,\sfT_z\opmu} \d z = \hd\Tr{\op} \hd\Tr{\opmu}.
		\end{equation*}
	\end{lem}
	
	\begin{proof}
		The convolution of Wigner transforms can be seen as a function-valued convolution of operators by writing
		\begin{equation}\label{eq:convolution_Wigner}
			\(f_{\opmu} * f_{\op}\)(z) = \hd \Tr{\op\,\sfT_z\opmu^{(-)}},
		\end{equation}
		where $\opmu^{(-)}$ is the operator with integral kernel $\opmu^{(-)}(x,y) = \opmu(-x,-y)$, that is the operator such that $f_{\opmu^{(-)}}(z) = f_{\opmu}(-z)$. This was actually used in~\cite{werner_quantum_1984} to define a function-valued convolution of two operators. Formula~\eqref{eq:convolution_Wigner} follows for any $z\in\Rdd$ by using the fact that $\hd\Tr{\op\,\op_2} = \intd f_{\op} \,f_{\op_2}$ for any Hilbert--Schmidt operators $\op$ and $\op_2$. By a Young's type inequality for the expression in~\eqref{eq:convolution_Wigner} proved in \cite[Proposition~3.2]{werner_quantum_1984}, $f_{\opmu} * f_{\op}$ is integrable whenever $\opmu$ and $\op$ are trace class, and so
		\begin{equation}\label{eq:integral_convolution}
			\intdd \hd \Tr{\op\,\sfT_z\opmu^{(-)}}\d z = \intdd \(f_{\opmu} * f_{\op}\)(z) \d z.
		\end{equation}
		To finish the proof we need to use the Fubini theorem. Notice first that it follows from the Cauchy--Schwarz inequality that for any $n>d$,
		\begin{align*}
			\intdd \n{f_{\op}} &\leq C_{n,\eps} \intdd \n{f_{\op}}^2 \(1+\eps^2\(\n{x}^{2n}+\n{\xi}^{2n}\)\)\d x\d\xi
			\\
			&\leq C_{n,\eps} \(\hd\Tr{\op^2} + \eps^2 \intdd \n{f_{\op}\n{x}^n}^2 + \n{f_{\op}\n{\xi}^n}^2 \d x\d\xi\).
		\end{align*}
		Since the function $f_{\op}\n{x}^n$ is the Wigner transform of the operator with integral kernel $2^{-n}\n{x+y}^n\op(x,y)$, it follows that
		\begin{equation*}
			\intdd \n{f_{\op}\n{x}^n}^2\d x\d \xi = \hd \iintd \n{2^{-n}\n{x+y}^n\op(x,y)}^2\d x\d y \leq \tfrac{1}{2}\, \hd\Tr{\n{\op\n{x}^n}^2}
		\end{equation*}
		where we used the convexity of $r\mapsto r^n$. Since changing $x$ by $\xi$ in the Wigner transform just amounts to conjugate the operator by a scaled Fourier transform, one obtains also $\Nrm{f_{\op}\n{\xi}^n}{L^2(\Rdd)} \leq \frac{1}{\sqrt{2}} \Nrm{\op\n{\opp}^n}{\L^2}$. Hence it follows from Identity~\eqref{eq:integral_convolution} and the Fubini theorem that
		\begin{equation*}
			\intdd \hd \Tr{\op\,\sfT_z\opmu^{(-)}}\d z = \intdd f_{\opmu} \intdd f_{\op} = \hd \Tr{\opmu}\, \hd\Tr{\op}.
		\end{equation*}
		To finish the proof, replace $\opmu$ by $\opmu^{(-)}$ and notice that $\Tr{\opmu} = \Tr{\opmu^{(-)}}$.
	\end{proof}
	
	\begin{proof}[Proof of Proposition~\ref{prop:Wh_bound_0}]
		By analogy with the classical case, but noting that the coupling needs to be a positive operator\footnote{The analogy with the classical case could also have lead us to look at $\opgam(z) := \sfT_z\op_{g_h}^{1/2}\,\op\,\sfT_z\op_{g_h}^{1/2}$. Then it is still true that $\hd\Tr{\opgam(z)} = \tilde{f}_{\op}$, however, a simple computation shows that in this case $\intdd \opgam(z)\d z = \ttildop$.}, we define for any $z\in\Rdd$,
		\begin{equation*}
			\opgam(z) := \op^{1/2}\,\sfT_z\op_{g_h}\,\op^{1/2},
		\end{equation*}
		where $\sfT_z\op_{g_h}$ could also be denoted $\sfT_z\op_{g_h} = \ket{\psi_z}\bra{\psi_z} = \sfT_z\tildop_{\delta_0}$ with the notations of Equation~\eqref{eq:Toeplitz}. Since $\op_{g_h}$ and $\op$ are trace class operators, $\opgam$ is trace class and one can use the cyclicity of the trace to get
		\begin{equation*}
			\hd\Tr{\opgam(z)} = \hd\Tr{\op\,\sfT_z\op_{g_h}} = (f_{\op}*g_h)(z) = \tilde{f}_{\op}(z)
		\end{equation*}
		where the second equality follows from Equation~\eqref{eq:convolution_Wigner} and the fact that $\op_{g_h}$ has an even kernel. On the other hand, taking the integral in the weak sense
		\begin{equation*}
			\intdd \opgam(z)\d z = \op^{1/2} \intdd \sfT_z\op_{g_h} \d z\, \op^{1/2} = \op,
		\end{equation*}
		because the family of operators $\sfT_z\op_{g_h}$ form a partition of unity. This proves that $\opgam$ is a coupling of $\tilde{f}_{\op}$ and $\op$.
		
		Now we compute the cost associated to this coupling. For any $z\in\Rdd$, $\opc = \opc(z)$ is an unbounded but densely defined self-adjoint operator. Since $\sqrt{\opc}\,\sqrt{\op}$ is a bounded operator and for any $z\in\Rdd$, $\sfT_z\op_{g_h}$ is trace class, we deduce that $\sqrt{\opc}\,\opgam(z)\, \sqrt{\opc} = \sqrt{\opc}\,\sqrt{\op}\,\sfT_z\op_{g_h}\sqrt{\op}\,\sqrt{\opc}$ is trace class and
		\begin{equation}\label{eq:def_Eps_z}
			\Eps_z(\op) := \hd\Tr{\opgam(z)^{1/2}\, \opc\,\opgam(z)^{1/2}} = \hd\Tr{\opc^{1/2}\,\opgam(z)\, \opc^{1/2}}.
		\end{equation}
		
		Recall the notation $\n{\opz} = \sqrt{\n{x}^2+\n{\opp}^2}$ where $\opp = -i\hbar\nabla$ and assume that the operator $\(1+\n{\opz}^n\)\sqrt{\op}$ is a trace class operator for $n$ sufficiently high so as to justify the cyclicity of the trace and the use of Lemma~\ref{lem:integral_convolution} in the next steps. In particular, $\opc\sqrt{\op} = \sfT_z\!\n{\opz}\sqrt{\op}$ is a bounded operator, hence since $\op_{g_h}$ is trace class, we deduce that $\opc\,\opgam(z)$ is trace class and by cyclicity of the trace
		\begin{multline*}
			\intdd\Eps_z(\op) \d z = \hd\intdd \Tr{\opc \,\opgam(z)}\d z = \hd\intdd \Tr{\sfT_z\!\n{\opz}^2\op^{1/2}\sfT_z\op_{g_h}\op^{1/2}}\d z
			\\
			= \hd\intdd \Tr{\op^{1/2}\sfT_z\!\(\n{\opz}^2\op_{g_h}\)\op^{1/2} + \com{\sfT_z\!\n{\opz}^2,\op^{1/2}}\sfT_z\op_{g_h} \op^{1/2}}\d z = I_1 + I_2.
		\end{multline*}
		The term $I_1$ is the term appearing in the classical case, while $I_2$ is an additional term coming from the non-commutativity. To treat the term $I_1$, use cyclicity of the trace and then Lemma~\ref{lem:integral_convolution} to get that
		\begin{multline}\label{eq:classical_term}
			I_1 = \hd\intdd \Tr{\op \,\sfT_z\!\(\n{\opz}^2\op_{g_h}\)}\d z= \hd\Tr{\op} \hd\Tr{\n{\opz}^2\op_{g_h}}
			\\
			= \hd\Tr{\(\n{x}^2+\n{\opp}^2\)\op_{g_h}} = \intdd \n{z}^2 g_h(z)\d z = d\hbar,
		\end{multline}
		where we used the fact that $\hd\Tr{\op} = 1$.
		% int z^2 e^(-z^2/hbar) dz = hbar^(d+1) int z^2 e^(-z^2) d z
		% = hbar^(d+1) om_2d int r^(2d+1) e^(-r^2) d r
		% = hbar^(d+1) om_2d /2 int t^(d) e^(-t) d t
		% = hbar^(d+1) om_2d /2 G(d+1)
		% = hbar^(d+1) (2 pi^d/G(d)) /2 G(d) d
		% = (h/2)^(d+1) d/pi
		
		Now we treat the term $I_2$. Denoting $z$ on the form $z = (y,\xi)$, by the Leibniz formula for commutators
		\begin{equation*}
			\com{\n{x-y}^2,\opmu} = \(x-y\)\cdot\com{x-y,\opmu} + \com{x-y,\opmu}\cdot\(x-y\)
		\end{equation*}
		and so since $y = y\,\Id_{L^2(\Rd)}$ commutes with $\opmu$, it follows from the definition of $\Dhv$ that
		\begin{equation*}
			\com{\n{x-y}^2,\opmu} = i\hbar\(\(x-y\)\cdot\Dhv\opmu + \Dhv\opmu\cdot\(x-y\)\)
		\end{equation*}
		and similarly, $\com{\n{\opp-\xi}^2,\opmu} = -i\hbar\(\(\opp-\xi\)\cdot\Dhx\opmu + \Dhx\opmu\cdot\(\opp-\xi\)\)$. One can summarize these two identities by writing
		\begin{equation*}
			\com{\sfT_z\!\n{\opz}^2,\sqrt{\op}} = i\hbar\(\sfT_z\opz \cdot \Dh^\perp\!\!\sqrt{\op} + \Dh^\perp\!\!\sqrt{\op}\cdot \sfT_z\opz\)
		\end{equation*}
		with $\Dh^\perp = (\Dhv,-\Dhx)$. This can also be written
		\begin{equation}\label{eq:com_z2_op}
			\com{\sfT_z\!\n{\opz}^2,\sqrt{\op}} = i\hbar\(\com{\sfT_z\opz; \Dh^\perp\!\!\sqrt{\op}} + 2\, \Dh^\perp\!\!\sqrt{\op}\cdot \sfT_z\opz\)
		\end{equation}
		where $\com{A;B} = A\cdot B - B \cdot A$. Multiplying the second term in the right-hand side of the above expression by $\sfT_z\op_{g_h}\sqrt{\op}$ and taking the integral of the trace yields by cyclicity of the trace and Lemma~\ref{lem:integral_convolution}
		\begin{align*}
			\intdd \Tr{\Dh^\perp\!\!\sqrt{\op}\cdot \sfT_z\opz\,\sfT_z\op_{g_h} \sqrt{\op}}\d z &= \hd\intdd \Tr{\sqrt{\op}\,\Dh^\perp\!\!\sqrt{\op}\cdot \sfT_z\!\(\opz\op_{g_h}\)}\d z
			\\
			&= \hd \Tr{\sqrt{\op}\,\Dh^\perp\!\!\sqrt{\op}} \cdot \Tr{\opz\op_{g_h}}.
		\end{align*}
		But since $g_h$ is even
		\begin{equation*}
			\hd \Tr{\opz\op_{g_h}} = \intdd z\,g_h(z)\d z = 0.
		\end{equation*}
		Hence it follows from Equation~\eqref{eq:com_z2_op} that
		\begin{equation*}
			I_2 = i\hbar \intdd \hd\Tr{\com{\sfT_z\opz; \Dh^\perp\!\!\sqrt{\op}} \sfT_z\op_{g_h} \sqrt{\op}}\d z.
		\end{equation*}
		Using the fact that $\com{\sfT_z\opz; \Dh^\perp\!\!\sqrt{\op}} = \com{\opz; \Dh^\perp\!\!\sqrt{\op}} = i\hbar\,\DDh\sqrt{\op}$ where $\DDh = \Dhx{\cdot\Dhx{}} + \Dhv{\cdot\Dhv{}}$, it follows again by the cyclicity of the trace and Lemma~\ref{lem:integral_convolution} that
		\begin{equation*}
			I_2 = -\hbar^2 \intdd \hd\Tr{\sqrt{\op}\,\DDh\sqrt{\op}\, \sfT_z\op_{g_h}}\d z = -\hbar^2\, \hd\Tr{\sqrt{\op}\,\DDh\sqrt{\op}}
		\end{equation*}
		where we used the fact that $\hd\Tr{\op_{g_h}} = 1$. Combined with Equation~\eqref{eq:classical_term}, this yields
		\begin{equation*}
			\intdd \Eps_z(\op) \d z = d\hbar-\hbar^2 \hd \Tr{\DDh\sqrt{\op} \sqrt{\op}}
			\\
			= d\hbar + \hbar^2 \hd \Tr{\n{\Dh{\sqrt{\op}}}^2}.
		\end{equation*}
		This proves the proposition for $\op$ sufficiently nice.
		
		More precisely, we proved that for any $\opmu = \sqrt{\op}$ sufficiently nice in the sense that $\(1+\n{\opp}^n\)\opmu$ is trace class, it holds
		\begin{equation}\label{eq:identity_0}
			\intdd \Eps_z(\op) \d z = \intdd \hd\Tr{\sqrt{\opc}\,\opmu\,\sfT_z\op_{g_h}\,\opmu\, \sqrt{\opc}}\d z = d\hbar + \hbar^2 \Nrm{\Dh{\opmu}}{\L^2}^2.
		\end{equation}
		Now the theorem follows by an approximation argument that we explain below. First notice that the left-hand side is defined as the integral of a positive function $\Eps_z(\op)$, which is well-defined as discussed before Equation~\eqref{eq:def_Eps_z} as soon as for any $z\in\Rdd$, $\sqrt{\opc}\,\opmu$ is a bounded operator. This is in particular the case when $\Nrm{\n{\opz}\opmu}{2} < \infty$ since for any $z=(y,\xi)\in\Rdd$,
		\begin{multline*}
			\Nrm{\sqrt{\opc}\opmu}{\infty}^2 \leq \Nrm{\sqrt{\opc}\opmu}{2}^2 = \Tr{\opmu\(\n{\xi-\opp}^2 + \n{y-x}^2\)\opmu}
			\\
			\leq 2\(\n{z}^2 \Nrm{\opmu}{2}^2 + \Nrm{\n{\opp}\opmu}{2}^2 + \Nrm{\n{x}\opmu}{2}^2\) = 2\(\n{z}^2\Tr{\op} + \Nrm{\n{\opz}\sqrt{\op}}{2}^2\)
		\end{multline*}
		as follows for instance using H\"older's inequality for Schatten norms. For the right-hand side of~\eqref{eq:identity_0}, the same result follows by using the triangle inequality on the commutators defining $\Dh$. For instance, $\hbar\Nrm{\Dhv{\opmu}}{2} = \Nrm{\com{x,\opmu}}{2} \leq 2 \Nrm{\n{\opz}\opmu}{2}$. Now diagonalizing $\opmu$ on the form
		\begin{equation*}
			\opmu = \sum_{j\in \N} \lambda_j \ket{\psi_j}\bra{\psi_j}
		\end{equation*}
		for some orthonormal basis $(\psi_j)_{j\in \N}$ of $L^2(\Rd)$ and a positive non-increasing family of eigenvalues $(\lambda_j)_{j\in\N}$, one notices that
		\begin{equation*}
			\Nrm{\opmu}{\L^2} + \Nrm{\n{\opz}\opmu}{2}^2 = \sum_{j\in\N} \lambda_j \intd \(1+\n{x}^2\) \n{\psi_j}^2 + \n{\hbar\nabla\psi_j}^2,
		\end{equation*}
		which can be seen as a sum of weighted Sobolev norms. By standard approximation arguments in Sobolev spaces, on can approximate each $\psi_j$ by a sequence of functions $\varphi_{j,k} \in C^\infty_c$ so that $\opmu_k = \sum \lambda_j \ket{\varphi_{j,k}}\bra{\varphi_{j,k}}$ satisfies Identity~\eqref{eq:identity_0} and $\Nrm{\n{\opz} (\opmu-\opmu_k)}{2} + \Nrm{\opmu-\opmu_k}{2} \to 0$ when $k\to \infty$. It yields in particular $\Nrm{\Dh\!\(\opmu-\opmu_k\)}{2} \to 0$. Moreover, since for any $z\in\Rdd$, $\Nrm{\sqrt{\opc}\(\opmu-\opmu_k\)}{2} \to 0$ it holds 
		\begin{multline*}
			\Eps_z(\op) - \Eps_z(\op_k) = \Nrm{\sqrt{\opc}\, \opmu \,\sfT_z\sqrt{\op_{g_h}}}{\L^2}^2 - \Nrm{\sqrt{\opc}\, \opmu_k \,\sfT_z\sqrt{\op_{g_h}}}{\L^2}^2
			\\
			= 2 \,\hd \Re{\Tr{\sqrt{\opc}\(\opmu-\opmu_k\)\sfT_z\op_{g_h}\,\opmu\, \sqrt{\opc}}} + \Nrm{\sqrt{\opc}\(\opmu-\opmu_k\)\sfT_z\sqrt{\op_{g_h}}}{\L^2}^2
		\end{multline*}
		and so we deduce that $\Eps_z(\op_k) \to \Eps_z(\op)$ when $k\to\infty$. Hence, by Equation~\eqref{eq:identity_0} and Fatou's lemma, we deduce that
		\begin{equation*}
			\intdd \hd\Tr{\opgam(z)^{1/2}\, \opc\,\opgam(z)^{1/2}} \d z \leq d\hbar + \hbar^2 \Nrm{\Dh{\sqrt{\op}}}{\L^2}^2.
		\end{equation*}
		Together with the fact that $\opgam$ is a coupling of $\op$ and $\tilde{f}_{\op}$, this proves the proposition by definition of the $\Wh$ distance.
	\end{proof}
	
	The proof of our main theorem follows straightforwardly from the above proposition.
	\begin{proof}[Proof of Theorem~\ref{thm:Wh_bound}]
		Inequality~\eqref{eq:Wh_bound} follows immediately from Proposition~\ref{prop:Wh_bound_0} using the fact that $\nabla f_{\sqrt{\op}} = f_{\Dh\sqrt{\op}}$ and the fact that the Wigner transform is an isometry from $\L^2$ to $L^2(\Rdd)$. Inequality~\eqref{eq:MKh_bound_0} then follows by the triangle inequality~\eqref{eq:triangle_MKh} in the form
		\begin{equation*}
			\MKh(\op,\op) \leq \Wh(\op,\tilde{f}_{\op}) + \Wh(\tilde{f}_{\op},\op) = 2 \Wh(\tilde{f}_{\op},\op)
		\end{equation*}
		which is bounded by Inequality~\eqref{eq:Wh_bound}.
	\end{proof}
	
	\begin{prop}\label{prop:MKh_vs_Wh}
		Let $\op\in \opP_2(\h)$ and $f\in \P_2(\Rdd)$, then
		\begin{equation*}
			\MKh(\tildop_f,\op)^2 \leq \Wh(f,\op)^2 + d \hbar.
		\end{equation*}
	\end{prop}
	
	\begin{remark}
		The proof of the "reverse" Inequality~\eqref{eq:Wh_vs_MKh} in~\cite[Theorem~5.4]{golse_optimal_2022} is done using a dual formulation of the quantum optimal transport. One can however also obtain a direct proof in the spirit of the proof below by considering $\opgam_2(z) = \hd \tr_1(\opgam^{1/2}(\sfT_z\op_{g_h}\otimes 1)\opgam^{1/2})$ which is a coupling of $\tilde{f}_{\op_1}$ and $\op_2$ whenever $\opgam\in\cC(\op_1,\op_2)$.
	\end{remark}
	
	\begin{proof}
		Let $\opgam \in \cC(f,\op)$ be a semiclassical coupling. Then we define $\opgam_2\in \opP(\h\otimes\h)$ by the weak integral
		\begin{equation*}
			\opgam_2 = \intd \sfT_z\op_{g_h}\otimes \opgam(z)\d z.
		\end{equation*}
		Taking the partial trace with respect to the first variable gives
		\begin{equation*}
			\hd \tr_1(\opgam_2) = \intd \hd \Tr{\sfT_z\op_{g_h}} \opgam(z) \d z = \intd \opgam(z) \d z = \op
		\end{equation*}
		using the fact that $\opgam\in\cC(f,\op)$ to get the last identity. On the other hand, taking the partial trace with respect to the second variable gives
		\begin{equation*}
			\hd \tr_2(\opgam_2) = \intd\hd\Tr{\opgam(z)} \sfT_z\op_{g_h}  \d z = \intd f(z)\, \sfT_z\op_{g_h} \d z = \tildop_f.
		\end{equation*}
		We deduce from these two equations that $\opgam_2\in\cC(\tildop_f,\op)$. As in the previous proofs, by an approximation argument, we can assume cyclicity of the trace and so it remains to compute the cost associated to this coupling, that is
		\begin{equation}\label{eq:cost_gamma_2}
			h^{2d} \Tr{\n{\opz_1-\opz_2}^2\opgam_2} = \intd h^{2d} \Tr{\n{\opz_1-\opz_2}^2\(\sfT_z\op_{g_h}\otimes \opgam(z)\)} \d z
		\end{equation}
		where $\opz_k = (x_k, \opp_k)$ with $\opp_k = -i\hbar\nabla_{\!x_k}$, so that $\n{\opz_1-\opz_2}^2 = \n{x_1-x_2}^2 + \n{\opp_1-\opp_2}^2$. Let us start by using Formula~\eqref{eq:convolution_Wigner} which gives the convolution of Wigner transforms to compute
		\begin{equation*}
			\hd\tr_1\!\(\n{x_1-x_2}^2\sfT_z\op_{g_h}\) = \intdd \n{x_1-x_2}^2 g_h(z_1-z)\d z_1 = \n{x-x_2}^2 + \frac{d\hbar}{2}.
		\end{equation*}
		By symmetry, an analogous result holds by replacing the operators $x_k$ by $\opp_k$ and $x$ by $\xi$, and so one obtains finally for any $z = (x,\xi)\in\Rdd$,
		\begin{equation*}
			\hd\tr_1\!\(\n{\opz_1-\opz_2}^2\sfT_z\op_{g_h}\) = \n{z-\opz_2}^2 + d\hbar.
		\end{equation*}
		Using the definition~\eqref{eq:def_MKh} of the quantum pseudometric and inserting the above identity in Formula~\eqref{eq:cost_gamma_2} yields
		\begin{equation*}
			\MKh(\tildop_f,\op)^2 \leq h^{2d} \Tr{\n{\opz_1-\opz_2}^2\opgam_2} = \intd \hd \Tr{\n{z-\opz_2}^2 \opgam(z)} \d z
		\end{equation*}
		where $\opgam$ and the trace in the second formula act on functions of the variable $x_2$. The result then follows by taking the supremum over all couplings $\opgam\in\cC(f,\op)$.
	\end{proof}
	
	\begin{proof}[Proof of Theorem~\ref{thm:MKh_bound}]
		Theorem~\ref{thm:MKh_bound} is a direct consequence of Formula~\eqref{eq:MKh_vs_Wh} applied to the case $f = \tilde{f}_{\op}$ and Theorem~\ref{thm:Wh_bound}.
	\end{proof}
	
\section{Comparison with negative order Sobolev norms}

\subsection{The upper bound}

	In this section we prove Theorem~\ref{thm:upper_bound} which provides an upper bound on the quantum optimal transport pseudometrics in terms of classical or quantum Sobolev norms. A more detailed version of the theorem is presented in the following proposition.

	\begin{prop}\label{prop:upper_bound}
		Let $n> 2$ and $(\op,\op_2)\in\cP_n(\h)^2$ uniformly with respect to $\hbar$. Then there exists a constant $C_{\op,\op_2}$ independent of $\hbar\in(0,1)$ such that
		\begin{align*}
			\Wh(f,\op) &\leq C_{f,\op} \Nrm{f - \tilde{f}_{\op}}{\dot{W}^{-1,1}(\Rdd)}^\frac{n-2}{2\(n-1\)} +  C_{\op} \,\sqrt{\hbar}
			\\
			\MKh(\op,\op_2) &\leq C_{\op,\op_2} \Nrm{\op - \op_2}{\dot{\cW}^{-1,1}}^\frac{n-2}{2\(n-1\)} +  \(C_{\op} + C_{\op_2}\) \sqrt{\hbar}.
		\end{align*}
		The following bounds hold
		\begin{align*}
			C_{\op} &\leq \sqrt{d + \hbar \Nrm{\Dh{\sqrt{\op}}}{\L^2}^2}
			\\
			C_{\op,\op_2} &\leq \frac{n-1}{\(n-2\)^\frac{n-2}{n-1}} \(P(\op)+P(\op_2)\)^\frac{n}{2n-2}
		\end{align*}
		with
		\begin{equation*}
			P(\op)^n = 3^{n-2} \hd \Tr{\sqrt{\op}\(\n{x}^n+\n{\opp}^n+\((d+n-2)\hbar\)^{n/2}\)\sqrt{\op}}.
		\end{equation*}
	\end{prop}
	% 0< c < 1
	% a^c + b^c < 2^(1-c) (a+b)^c
	% a^c + b^c + d^c < 3^(1-c) (a+b+d)^c
	% (x^(2/n)+v^(2/n)+ch)^(1/2) + (xx^(2/n)+vv^(2/n)+cch)^(1/2)
	% < 2^(1/2) (x^(2/n)+v^(2/n)+ch+xx^(2/n)+vv^(2/n)+cch)^(1/2)
	% < 2^(1/2) 5^(1-2/n) (x+v+(c+cc)^(n/2)h^(n/2)+xx+vv)^(1/n)
	%
	% (x^(2/n)+v^(2/n)+ch)^(1/2) + (xx^(2/n)+vv^(2/n)+cch)^(1/2)
	% < 3^(1-2/n) ((x+v+(ch)^(n/2))^(1/n) + (xx+vv+(cch)^(n/2))^(1/n))
	
	We first begin with the following lemma on classical Wasserstein distances, which is just a quantitative estimate of the fact that weak convergence of measures implies convergence in the $\W_2$ distance when sufficiently many moments are bounded.
	\begin{lem}\label{lem:classical_weak_estimate}
		Let $n>2$, then for any probability measures $f_1$ and $f_2$ on $\Rdd$, it holds
		\begin{equation*}
			\W_2(f_1,f_2) \leq C_n \(Z_n(f_1) + Z_n(f_2)\)^\theta W_1(f_1,f_2)^{1-\theta}
		\end{equation*}
		where $Z_n(f)^n = \intdd \n{z}^n f(\d z)$, $\theta = \frac{n}{2\(n-1\)}$ and $C_n = \(n-1\)\(n-2\)^\frac{2-n}{n-1}$. In particular, it follows from the Kantorovich--Rubinstein theorem that
		\begin{equation*}
			\W_2(f_1,f_2) \leq C_n \(Z_n(f_1) + Z_n(f_2)\)^\theta \Nrm{f_1-f_2}{\dot{W}^{-1,1}}^{1-\theta}.
		\end{equation*}
	\end{lem}
	
	\begin{proof}
		Let $\gamma$ be the optimal transport associated to $f_1$ and $f_2$, so that
		\begin{equation*}
			\W_2(f_1,f_2)^2 = \int_{\R^{4d}} \n{z_1-z_2}^2 \gamma(\d z_{12}),
		\end{equation*}
		where $\d z_{12} = \d z_1 \d z_2$. Then for any $R>0$, by the triangle inequality
		\begin{align*}
			\W_2(f_1,f_2) &\leq \(\int_{\n{z_1-z_2}\leq R} \n{z_1-z_2}^2 \gamma(\d z_{12})\)^\frac{1}{2} + \(\int_{\n{z_1-z_2}\geq R} \n{z_1-z_2}^2 \gamma(\d z_{12})\)^\frac{1}{2}
			\\
			&\leq \sqrt{R}\(\int_{\R^{4d}} \n{z_1-z_2} \gamma(\d z_{12})\)^\frac{1}{2} + R^\frac{2-n}{2}\(\int_{\R^{4d}} \n{z_1-z_2}^n \gamma(\d z_{12})\)^\frac{1}{2}
		\end{align*}
		Using the definition of the $W_1$ distance to control the first term and the triangle inequality to get a bound on the second term, we deduce that for any $r = \sqrt{R}>0$,
		\begin{align*}
			\W_2(f_1,f_2) &\leq r \,W_1(f_1,f_2)^\frac{1}{2} + r^{2-n} \(\Nrm{f_1\n{z}^n}{L^1}^\frac{1}{n} + \Nrm{f_2\n{z}^n}{L^1}^\frac{1}{n}\)^\frac{n}{2}
		\end{align*}
		where we used the fact that $\gamma$ is a coupling of $f_1$ and $f_2$. The result then follows by optimizing with respect to $r$.
	\end{proof}
	
	We now use the above lemma to get its quantum analogue.
	\begin{proof}[Proof of Proposition~\ref{prop:upper_bound}]
		Notice that for any $\varphi \in \dot{W}^{1,\infty}$, since the Wick quantization is non-expansive from $L^\infty(\Rdd)$ to $\L^\infty$
		\begin{equation*}
			\Nrm{\tildop_\varphi}{\dot{\cW}^{1,\infty}} = \Nrm{\Dh \tildop_\varphi}{\L^\infty} = \Nrm{\tildop_{\nabla\varphi}}{\L^\infty} \leq \Nrm{\nabla\varphi}{L^\infty} = \Nrm{\varphi}{\dot{W}^{1,\infty}}.
		\end{equation*}
		Hence, we deduce that
		\begin{equation*}
			\intdd \tilde{f}_{\opmu} \,\varphi = \hd\Tr{\opmu\,\tildop_{\varphi}} \leq \sup_{\Nrm{\op}{\dot{\cW}^{1,\infty}\leq 1}} \hd\Tr{\opmu\,\op} = \Nrm{\opmu}{\dot{\cW}^{-1,1}}
		\end{equation*}
		that is, $\Nrm{\tilde{f}_{\opmu}}{\dot{W}^{-1,1}} \leq \Nrm{\opmu}{\dot{\cW}^{-1,1}}$. In particular, taking $\opmu = \op - \op_2$, it follows from the linearity of the Husimi transform that
		\begin{equation}\label{eq:husimi_vs_op_weak}
			\Nrm{\tilde{f}_{\op} - \tilde{f}_{\op_2}}{\dot{W}^{-1,1}} \leq \Nrm{\op - \op_2}{\dot{\cW}^{-1,1}}.
		\end{equation}
		On another side, notice that it follows from the positivity of $\tilde{f}_{\op}$ and the triangle inequality that
		\begin{equation}\label{eq:splitting_moments}
			Z_n(\tilde{f}_{\op})^2 \leq \(\intdd \n{x}^n \tilde{f}_{\op}(\d z)\)^\frac{2}{n} + \(\intdd \n{\xi}^n \tilde{f}_{\op}(\d z)\)^\frac{2}{n}
		\end{equation}
		where we use the notation $z = (x,\xi)$. Since $\tilde{f}_{\op} = e^{\hbar\Delta_x/4}e^{\hbar\Delta_\xi/4} f_{\op}$, defining $\rho(x) = \intd f_{\op}\d \xi$, we see that
		\begin{equation*}
			\intdd \n{x}^n \tilde{f}_{\op}(\d z) = \intd \n{x}^n e^{\hbar\Delta/4}\rho(\d x).
		\end{equation*}
		Defining $u(t,x)$ the solution of $\dpt u = \Delta u$ with initial condition $u(0,x) = \rho(x)$, and $m_n(t) := \intd u(t,x)\n{x}^n\d x$, it follows by integration by parts that $\dpt m_n = n\(d+n-2\) m_{n-2}$. By H\"older's inequality, this yields
		\begin{equation*}
			\dpt m_n \leq n\(d+n-2\) m_0^{2/n} m_n^{1-2/n}
		\end{equation*}
		where $m_0 = m_0(t) = m_0(0) = \intd \rho$. It follows from Gr\"onwall's inequality that $m_n(t)^{2/n} \leq m_n(0)^{2/n} + \(\frac{d+n}{2}-1\)\hbar \,m_0^{2/n}$, which for $t=\hbar/4$ yields
		\begin{equation*}
			\(\intd \n{x}^n e^{\hbar\Delta/4}\rho(\d x)\)^\frac{2}{n} \leq \(\intd \n{x}^n \rho(\d x)\)^\frac{2}{n} + \(\tfrac{d+n}{2}-1\)\hbar\, m_0^{2/n}.
		\end{equation*}
		Recalling the well-known property that $\intd \n{x}^n \rho(\d x) = \hd\Tr{\op \n{x}^n}$ and performing the same analysis for the second term in the right-hand side of Inequality~\eqref{eq:splitting_moments}, one deduces
		\begin{equation*}
			Z_n(\tilde{f}_{\op})^2 \leq N_n(\op)^\frac{2}{n} + M_n(\op)^\frac{2}{n} + \(d+n-2\)\hbar\, m_0^\frac{2}{n}
		\end{equation*}
		where $N_n(\op) = \hd\Tr{\sqrt{\op} \n{x}^n \sqrt{\op}}$ and $M_n(\op) = \hd\Tr{\sqrt{\op} \n{\opp}^n \sqrt{\op}}$. We can now use Lemma~\ref{lem:classical_weak_estimate} with $f_1 = \tilde{f}_{\op}$ and $f_2 = \tilde{f}_{\op_2}$ to get
		\begin{equation}\label{eq:W2_Husimi_vs_cW_op}
			\W_2(\tilde{f}_{\op},\tilde{f}_{\op_2}) \leq C_n \(Z_n(\op) + Z_n(\op_2)\)^\theta \Nrm{\op - \op_2}{\dot{\cW}^{-1,1}}^{1-\theta}
		\end{equation}
		where $Z_n(\op)^2 = N_n(\op)^\frac{2}{n} + M_n(\op)^\frac{2}{n} + \(d+n-2\)\hbar\, m_0^\frac{2}{n}$. To finish the proof, it just remains to use the triangle inequalities~\eqref{eq:triangle_MKh} and~\eqref{eq:triangle_Wh} to get
		\begin{align*}
			\MKh(\op,\op_2) &\leq \Wh(\op,\tilde{f}_{\op}) + \Wh(\tilde{f}_{\op},\op_2)
			\\
			&\leq \Wh(\op,\tilde{f}_{\op}) + \W_2(\tilde{f}_{\op},\tilde{f}_{\op_2}) + \Wh(\tilde{f}_{\op_2},\op_2)
		\end{align*}
		which by Theorem~\ref{thm:Wh_bound} and Inequality~\eqref{eq:W2_Husimi_vs_cW_op} yields the result.
	\end{proof}

\subsection{The lower bound}

	In this section we prove Theorem~\ref{thm:lower_bound}, which gives a lower bound on the quantum optimal transport pseudometrics in terms of quantum Sobolev norms. The next proposition proves the second inequality of this theorem, and the first inequality follows by a similar proof.

	\begin{prop}
		Let $\op,\op_2\in\cP_2(\h)$. Then for any $d>1$ and $s\in[-1,1]$,
		\begin{equation*}
			\Nrm{\op-\op_2}{\dot{\cH}^{-1}} -  \(\tfrac{\hbar}{4}\)^\frac{s+1}{2} \Nrm{\op-\op_2}{\dot{\cH}^s} \leq \cC_\infty^{1/2} \(\MKh(\op,\op_2)^2 + 2d\hbar\)^\frac{1}{2}
		\end{equation*}
		where $\cC_\infty = \max\!\(\Nrm{\op}{\L^\infty}, \Nrm{\op_2}{\L^\infty}\)$.
	\end{prop}

	\begin{proof}
		We know by \cite{loeper_uniqueness_2006} that
		\begin{equation*}
			\Nrm{f - f_2}{\dot{H}^{-1}} \leq \sqrt{\max(\Nrm{f}{L^\infty},\Nrm{f_2}{L^\infty})}\, \W_2(f,f_2)
		\end{equation*}
		where $\dot{H}^{-1} = \dot{H}^{-1}(\Rdd)$. Applying this inequality to the Husimi transforms $f = \tilde{f}_{\op}$ and $f_2 = \tilde{f}_{\op_2}$ and using the well-known fact that $\sNrm{\tilde{f}_{\op}}{L^\infty(\Rdd)} \leq \Nrm{\op}{\L^\infty}$ and Inequality~\eqref{eq:W2_Husimi_vs_MKh} leads to the following inequality
		\begin{equation*}
			\Nrm{\tilde{f}_{\op} - \tilde{f}_{\op_2}}{\dot{H}^{-1}}^2 \leq \max\!\(\Nrm{\op}{\L^\infty}, \Nrm{\op_2}{\L^\infty}\) \(\MKh(\op,\op_2)^2 + 2d\hbar\).
		\end{equation*}
		Since the function $r\mapsto \frac{1-e^{-r}}{r^\alpha}$ is smaller than $1$ on $\R_+$ for any $\alpha\in (0,1)$, it follows that for any $s\in[-1,1]$,
		\begin{equation*}
			\Nrm{\tilde{g} - g}{\dot{H}^{-1}} = \(\tfrac{\hbar}{4}\)^\frac{s+1}{2} \Nrm{\frac{1-e^{\hbar\Delta/4}}{(-\hbar\Delta/4)^{(s+1)/2}} (-\Delta)^{s/2}g}{L^2} \leq \(\tfrac{\hbar}{4}\)^\frac{s+1}{2} \Nrm{g}{\dot{H}^s}
		\end{equation*}
		we deduce that
		\begin{align*}
			\Nrm{\op-\op_2}{\dot{\cH}^{-1}} &= \Nrm{f_{\op} - f_{\op_2}}{\dot{H}^{-1}} \leq \Nrm{\tilde{f}_{\op} - \tilde{f}_{\op_2}}{\dot{H}^{-1}} + \Nrm{f_{\op-\op_2} - \tilde{f}_{\op-\op_2}}{\dot{H}^{-1}}
			\\
			&\leq \cC_\infty^\frac{1}{2} \(\MKh(\op,\op_2)^2 + 2d\hbar\)^\frac{1}{2} + \(\tfrac{\hbar}{4}\)^\frac{s+1}{2} \Nrm{f_{\op-\op_2}}{\dot{H}^s}
		\end{align*}
		The result then follows from the fact that the Wigner transform is an isometry from $\dot{\cH}^s$ to $\dot{H}^s(\Rdd)$.
	\end{proof}

\section{Proofs of the particular cases}

\subsection{Powers of Toeplitz operators}

	\begin{proof}[Proof of Proposition~\ref{prop:toplitz_power}]
		We want to apply Theorem~\ref{thm:Wh_bound} to $\op$. Notice first that $\hd \Tr{\op} = C_f \sNrm{\tildop_f}{\L^n}^n = 1$, so $\op\in\cP(\h)$. It remains to prove that $\op$ has two bounded moments. Let us for example look at the moments in $x$, for which
		\begin{align*}
			\hd\Tr{\sqrt{\op} \n{x}^2 \sqrt{\op}} &= C_f \hd\Tr{\tildop_f^\frac{n-2}{2}\, \tildop_f \n{x}^2 \tildop_f\, \tildop_f^\frac{n-2}{2}}
			\\
			&\leq C_f\, \hd\Tr{\tildop_f \n{x}^2 \tildop_f} \sNrm{\tildop_f}{\L^\infty}^{n-2}.
		\end{align*}
		Using the fact that the Wick quantization is a non-expansive mapping from $L^\infty(\Rdd)$ to $\L^\infty$ and the definition of the Wick quantization yields
		\begin{multline*}
			\hd\Tr{\sqrt{\op} \n{x}^2 \sqrt{\op}} 
			\\
			\leq  \frac{C_f}{h^{2d}} \Nrm{f}{L^\infty(\Rdd)}^{n-2}  \iint_{\R^{4d}} f(z)\,f(z')\,\Inprod{\psi_z}{\psi_{z'}} \intd \conj{\psi_z'(y)}\,\n{y}^2\psi_z(y)\d y\d z\d z'.
		\end{multline*}
		A tedious but straightforward computation using Gaussian integrals shows that this quantity is finite for every fixed $\hbar$ as soon as $f\in L^1\cap L^\infty$. A similar reasoning also holds by replacing $\n{x}^2$ by $\n{\opp}^2$. Hence $\op\in\cP_2(\h)$.
		
		Let $k = n/2$. To bound the right-hand side of Inequality~\eqref{eq:MKh_bound_0}, we need to obtain an estimate on
		\begin{equation*}
			\Nrm{\Dh\sqrt{\op}}{\L^2} = C_f^{1/2} \Nrm{\Dh\tildop_f^k}{\L^2} = C_f^{1/2}  \Nrm{\sum_{j=1}^k\tildop^{j-1}\(\Dh\tildop_f\)\tildop^{k-j}}{\L^2}.
		\end{equation*}
		By the triangle inequality, H\"older's inequality for Schatten norms and the fact that the Weyl quantization is an isometry from $\L^2$ to $L^2(\Rdd)$, one gets
		\begin{equation*}
			\Nrm{\Dh\sqrt{\tildop_f^n}}{\L^2} \leq \sum_{j=1}^k \Nrm{\tildop_f^{j-1}}{\L^\infty} \Nrm{\Dh\tildop_f}{\L^2} \Nrm{\tildop_f^{k-j}}{\L^\infty} \leq k \Nrm{f}{L^\infty(\Rdd)}^{k-1} \Nrm{\nabla \tilde{f}}{L^2(\Rdd)}.
		\end{equation*}
		Now notice that
		\begin{equation*}
			\hbar^2\Nrm{\nabla \tilde{f}}{L^2(\Rdd)}^2 = h \intdd h\n{z}^2 e^{-\pi h\n{z}^2} \n{\widehat{f}(z)}^2\d z \leq \frac{h}{\pi\,e} \Nrm{f}{L^2}^2
		\end{equation*}
		since for any $r>0$, $r\, e^{-\pi r} \leq \frac{1}{\pi\,e}$. This proves Inequality~\eqref{eq:power_toplitz}. It remains to prove Equation~\eqref{eq:Cf_bound}. Since $\hbar \leq 1$
		\begin{equation*}
			B_f := \intdd \n{\widehat{f}(z)}^2 e^{-\pi \n{z}^2} \d z \leq \intdd \n{\widehat{f}(z)}^2 e^{-\pi h\n{z}^2} \d z = \Nrm{g_h * f}{L^2(\Rdd)}^2 = \Nrm{\tildop_f}{\L^2}^2.
		\end{equation*}
		When $n=2$, this concludes the proof since $C_f = \sNrm{\tildop_f}{\L^2}^{-2}$. When $n>2$, by H\"older's inequality for Schatten norms
		\begin{equation*}
			B_f \leq \Nrm{\tildop_f}{\L^n}^{n'} \Nrm{\tildop_f}{\L^1}^{2-n'} = C_f^\frac{-1}{n-1} \(\intdd f\)^{2-n'}
		\end{equation*}
		where $n' = n/(n-1)$.
	\end{proof}
	
\subsection{Estimates for projection operators}

	Since self-adjoint projection operators verify $\sqrt{\op} = \op$, Proposition~\ref{prop:pure_states} follows from Theorem~\ref{thm:Wh_bound} and the following bound for $\Nrm{\nabla\op}{\L^2}$.

	\begin{lem}\label{lem:pure_states}
		Assume that $\op$ is of the form given by Equation~\eqref{eq:spectral_proj} with $V$ satisfying~\eqref{eq:conditions_u} and $V_- := \max(0,-V)\in L^\infty$. Then there exists a constant $C_V>0$ independent of $\hbar\in(0,1)$ such that
		\begin{equation*}
			\Nrm{\nabla\op}{\L^2} \leq \frac{C_V}{\sqrt{\hbar}}.
		\end{equation*}
		The constant verifies for any $N>0$, $C_V^2 \leq 2\,(\Nrm{V_-}{L^\infty(\Omega)} + R + C\,\hbar^N)\Nrm{\op}{\dot{\cW}^{1,1}}$ where $R = \sup_{x\in\Omega}{\n{x}}$ and $C$ is a constant  independent of $\hbar$.
	\end{lem}
	
	\begin{proof}[Proof of Proposition~\ref{lem:pure_states}]
		We first look at the  quantum gradient with respect to $x$. It follows by its definition and H\"older's inequality that
		\begin{equation*}
			\Nrm{\Dhx{\op}}{\L^2} \leq \Nrm{\com{\nabla,\op}}{\L^\infty}^{1/2} \Nrm{\com{\nabla,\op}}{\L^1}^{1/2} \leq \frac{\sqrt{2}}{\sqrt{\hbar}}\, \Nrm{\opp\,\op}{\L^\infty}^{1/2} \Nrm{\Dhx\op}{\L^1}^{1/2}
		\end{equation*}
		where the last inequality follows from the triangle inequality for Schatten norms and the definition of $\opp = -i\hbar\nabla$. We check that $\Nrm{\opp\,\op}{\L^\infty}$ is bounded independently of $\hbar$. Notice that
		\begin{equation*}
			\n{\opp\,\op}^2 = \op\(\n{\opp}^2 + V(x)\)\op - \op\, V(x)\,\op
		\end{equation*}
		and so since $\op = \indic_{(-\infty,0]}(\n{\opp}^2 + V(x))$, we deduce that
		\begin{equation*}
			-\op\(\n{\opp}^2+V(x)\)\op = \(\n{\opp}^2 + V(x)\)_-
		\end{equation*}
		is a positive operator, and so since $0\leq \op\leq 1$
		\begin{equation*}
			0\leq \n{\opp\,\op}^2 \leq \op \(-V(x)\)_- \op \leq \Nrm{V_-}{L^\infty}.
		\end{equation*}
		Since $V_- = 0$ on $\Omega^c$ and it yields $\Nrm{\opp\,\op}{\L^\infty} \leq \Nrm{V_-}{L^\infty(\Omega)}^{1/2}$, which is bounded since $V\in C^\infty(\Omega)$. On the other side, it follows from~\cite[Theorem~1.2]{fournais_optimal_2020} that $\op\in\cW^{1,1}$ uniformly in $\hbar$, hence it follows that
		\begin{equation*}
			\Nrm{\Dhx{\op}}{\L^2} \leq \frac{C_1}{\sqrt{\hbar}}
		\end{equation*}
		where $C_1 \leq \sqrt{2} \Nrm{V_-}{L^\infty}^{1/2} \Nrm{\op}{\dot{\cW}^{1,1}}^{1/2}$ is bounded uniformly with respect to $\hbar$.
		
		The quantum gradient with respect to $\xi$ similarly yields the following inequality
		\begin{equation*}
			\Nrm{\Dhv{\op}}{\L^2} \leq \frac{\sqrt{2}}{\sqrt{\hbar}}\, \Nrm{x\,\op}{\L^\infty}^{1/2} \Nrm{\Dhv\op}{\L^1}^{1/2}.
		\end{equation*}
		Again, the last factor on the right hand-side is bounded using the fact that $\op\in\cW^{1,1}$ uniformly in $\hbar$. The fact that $\Nrm{x\,\op}{\L^\infty}^{1/2}$ is bounded uniformly with respect to $\hbar$ follows from Agmon's estimates~\cite{agmon_lectures_1982}, which imply (see e.g. \cite[Equation~(4.7)]{fournais_optimal_2020}) that for any $N\in\N$ and any $\chi\in C^\infty(\Rd)$ verifying $\indic_{\Omega_\eps} \leq \chi \leq \indic_{\Omega}$
 		\begin{equation*}
 			\Tr{\n{x\(1-\chi(x)\) \op}} \leq C\,\hbar^N
 		\end{equation*}
 		with $C$ independent of $\hbar$. Since the operator norm is bounded by the trace norm, $\Nrm{\op}{\infty} = 1$ and $\n{\chi(x)\, x} \leq R$, it follows from the triangle inequality that
 		\begin{equation*}
 			\Nrm{x\,\op}{\L^\infty} \leq R + C\,\hbar^N.
 		\end{equation*}
 		The result then follows from the identity $\Nrm{\Dh\op}{\L^2}^2 = \Nrm{\Dhx\op}{\L^2}^2  + \Nrm{\Dhv\op}{\L^2}^2$.
	\end{proof}
	
\subsection{Estimates for thermal states}

	In this section, we prove Proposition~\ref{prop:thermal}. As in the two previous sections, we start by computing the size of the moments and the commutators of thermal states and conclude using Theorem~\ref{thm:Wh_bound}.

	\begin{lem}\label{lem:thermal_moments}
		Let $H = -\hbar^2\Delta + V$ with $V$ satisfying Hypothesis~\eqref{eq:V_coercivity}. Then for any $n>0$ and $\beta>0$,
		\begin{align}\label{eq:thermal_velocity_L2}
			\Nrm{\n{\opp}^ne^{-\beta\,H}}{\L^2}^2 &\leq \frac{C_a\,n^n}{\kappa^{\frac{d}{a}} \,\beta^{n+d\(\frac{1}{a}+\frac{1}{2}\)}}
			\\\label{eq:thermal_position_L2}
			\Nrm{\n{x}^n e^{-\beta\,H}}{\L^\infty}^2 &\leq \frac{C_a \max\!\((2n\kappa\hbar)^\frac{2\,n}{1+a}, (\tfrac{2n}{\kappa a\beta})^\frac{n}{a}\)}{\kappa^\frac{d}{a}\,\beta^{d\(\frac{1}{a}+\frac{1}{2}\)}}
		\end{align}
		where $C_a = \frac{2\,\Gamma(d/a)}{a\,\Gamma(d/2)} \,\pi^d$.
	\end{lem}
	
	\begin{proof}
		Let $\op_\beta = e^{-\beta\,H}$. Let $\psi = \psi(t,x)$ be the solution of $\dpt \psi = -H\psi$ with $\psi(0,x) = \varphi(x)$, so that $\psi = e^{-tH}\varphi$, and let $y = \|\Delta^{n/2}\psi\|_{L^2}^{2/n}$ with $\Delta^s = -\(-\Delta\)^s$. Then
		\begin{equation*}
			\dpt y^n = 2\intd \(\Delta^\frac{n}{2}\psi\) \(\hbar^2\Delta-V\)\!\(\Delta^\frac{n}{2}\psi\) = -2\intd \hbar^2 \n{\Delta^{\frac{n+1}{2}}\psi}^2 + \n{\Delta^\frac{n}{2}\psi}^2 V.
		\end{equation*}
		By interpolation between Sobolev norms and using the fact that $c^{n/2} := \Nrm{\varphi}{L^2} \geq \Nrm{\psi}{L^2}$, we obtain
		\begin{equation*}
			\Nrm{\Delta^{\frac{n}{2}}\psi}{L^2}^2 \leq \Nrm{\Delta^{\frac{n+1}{2}}\psi}{L^2}^\frac{2\,n}{n+1} \Nrm{\psi}{L^2}^\frac{2}{n+1} \leq \Nrm{\Delta^{\frac{n+1}{2}}\psi}{L^2}^\frac{2\,n}{n+1} c^\frac{n}{n+1}.
		\end{equation*}
		Since $V\geq 0$, we deduce that $\dpt y^n \leq -2\,\hbar^2\,c^{-1}\,y^{n+1}$, or equivalently,
		\begin{equation*}
			\dpt y \leq -\frac{2\,\hbar^2}{n\,c}\,y^2.
		\end{equation*}
		This implies that
		\begin{equation*}
			\Nrm{\Delta^{n/2}e^{-tH}\varphi}{L^2}^{2/n} = y(t) \leq \frac{n\,c\,y(0)}{n\,c + 2\,\hbar^2\,t\, y(0)} \leq \frac{n}{2\,\hbar^2\,t} \Nrm{\varphi}{L^2}^{2/n}.
		\end{equation*}
		Recalling that $\n{\opp}^n = \hbar^n \(-\Delta\)^{n/2}$ and taking $t=\beta$ and the supremum over all $\varphi\in L^2$, we get
		\begin{equation}\label{eq:thermal_velocity_Linfty}
			\Nrm{\n{\opp}^n \op_\beta}{\L^\infty} \leq \(\frac{n}{2\,\beta}\)^{n/2}.
		\end{equation}
		On the other hand, by the Golden--Thomson inequality (see e.g. \cite[Section~8.1]{simon_trace_2005})
		\begin{equation*}
			\Nrm{\op_\beta}{\L^2}^2 = \hd\Tr{e^{-2\beta H}} \leq \intdd e^{-2\beta\(\n{\xi}^2+V(x)\)}\d x\d \xi
		\end{equation*}
		and so by Hypothesis~\eqref{eq:V_coercivity},
		\begin{equation}\label{eq:thermal_L2}
			\Nrm{\op_\beta}{\L^2}^2 \leq \intdd e^{-2\beta \(\n{\xi}^2+\kappa\n{x}^a\)}\d x\d \xi = C_a \(2\,\beta\)^{-d\(\frac{1}{a}+\frac{1}{2}\)}\,\kappa^{-d/a}.
		\end{equation}
		Combining Inequality~\eqref{eq:thermal_L2} and Inequality~\eqref{eq:thermal_velocity_Linfty} gives
		\begin{equation*}
			\Nrm{\n{\opp}^n\op_{\beta}}{\L^2}^2 \leq \Nrm{\n{\opp}^n\op_{\beta/2}}{\L^\infty}^2 \Nrm{\op_{\beta/2}}{\L^2}^2 \leq C_a\,\kappa^{-\,\frac{d}{a}} \(\frac{n}{\beta}\)^n \(\frac{1}{\beta}\)^{d\(\frac{1}{a}+\frac{1}{2}\)}.
		\end{equation*}
		This gives Formula~\eqref{eq:thermal_velocity_L2}. Inequality~\eqref{eq:thermal_position_L2} is obtained following the same strategy. By the same proof as in~\cite[Lemma~3]{chong_semiclassical_2023} one obtains from Condition~\eqref{eq:V_coercivity} that
		\begin{equation*}
			\Nrm{\n{x}^n \op_\beta}{\L^\infty} \leq \max\!\((2n\kappa\hbar)^\frac{2}{1+a}, (\tfrac{n}{\kappa a\beta})^\frac{1}{a}\)^{n/2}
		\end{equation*}
		and the result follows by combining this estimate with Inequality~\eqref{eq:thermal_L2}.
	\end{proof}
	
	We can now estimate the gradient of the square root of the thermal state.
	\begin{lem}\label{lem:thermal_gradient}
		Let $H = -\hbar^2\Delta + V$ with $V$ satisfying hypotheses~\eqref{eq:V_coercivity} and \eqref{eq:V_regu}. Then for any $\beta>0$,
		\begin{equation*}
			\Nrm{\Dh\sqrt{e^{-\beta H}}}{\L^2}^2 \leq \frac{C_{d,a}}{\kappa^{\frac{d}{a}} \beta^{d\(\frac{1}{a}+\frac{1}{2}\)-1}} \(2^4 + \kappa_2\,\beta \max\!\((2b\kappa\hbar)^\frac{2b}{1+a}, (\tfrac{8b}{\kappa a\beta})^\frac{b}{a}\)\)
		\end{equation*}
		where $C_{d,a} = 2^{2d\(\frac{1}{a}+\frac{1}{2}\)+1}\,\frac{\Gamma(d/a)}{a\,\Gamma(d/2)} \,\pi^d$.
	\end{lem}

	\begin{proof}
		Let $\op_\beta = e^{-\beta\,H}$. As in \cite{chong_semiclassical_2023} and using the fact that $\op_\beta^s = \op_{\beta s}$, it holds
		\begin{equation*}
			\Dh \sqrt{\op_\beta} = -\beta \int_0^1 \op_{\beta\(1-s\)/2}\, \Dh H\,\op_{\beta s/2} \d s
		\end{equation*}
		where $\Dh H = \(\Dhx H,\Dhv H\) = \(\nabla V,2\,\opp\)$. In particular, by invariance of the Schatten norms by taking the adjoint
		\begin{equation*}
			\Nrm{\Dhv \sqrt{\op_\beta}}{\L^2} \leq 2\,\beta \int_0^1 \Nrm{\op_{\beta\(1-s\)/2}\, \opp\,\op_{\beta s/2}}{\L^2} \d s = 4\,\beta \int_{1/2}^1 \Nrm{\op_{\beta\(1-s\)/2}\, \opp\,\op_{\beta s/2}}{\L^2} \d s.
		\end{equation*}
		Using the fact that $0\leq \op_{\beta\(1-s\)/2}\leq 1$ and Inequality~\eqref{eq:thermal_velocity_L2}, it gives
		\begin{equation*}
			\Nrm{\Dhv \sqrt{\op_\beta}}{\L^2} \leq 4\,\beta \int_{1/2}^1 \Nrm{\n{\opp}\op_{\beta s/2}}{\L^2} \d s \leq \frac{2^{2+d\(\frac{1}{a}+\frac{1}{2}\)} \sqrt{C_a}}{\kappa^{\frac{d}{2a}} \beta^{\frac{d}{2}\(\frac{1}{a}+\frac{1}{2}-\frac{1}{d}\)}}.
		\end{equation*} 
		Similarly, to bound the gradient with respect to $x$, we use the fact that $\n{x}^{-b}\,\nabla V$ is a bounded operator by Hypothesis~\eqref{eq:V_regu}, to get
		\begin{multline*}
			\Nrm{\Dhx \sqrt{\op_\beta}}{\L^2} \leq 2\,\beta\,\kappa_2 \int_{1/2}^1 \Nrm{\n{x}^b \op_{\beta s/2}}{\L^2} \d s 
			\\
			\leq \frac{2^{d\(\frac{1}{a}+\frac{1}{2}\)} \kappa_2\, \sqrt{C_a}}{\kappa^{\frac{d}{2a}} \beta^{\frac{d}{2}\(\frac{1}{a}+\frac{1}{2}\)-1}} \max\!\((2b\kappa\hbar)^\frac{b}{1+a}, (\tfrac{8b}{\kappa a\beta})^\frac{b}{2a}\).
		\end{multline*}
		The result follows by using the fact that $\Nrm{\Dh \sqrt{\op_\beta}}{\L^2}^2 = \Nrm{\Dhx \sqrt{\op_\beta}}{\L^2}^2 + \Nrm{\Dhv \sqrt{\op_\beta}}{\L^2}^2$.
	\end{proof}
	
	The result of Proposition~\ref{prop:thermal} now follows from the two above lemmas.
	\begin{proof}[Proof of Proposition~\ref{prop:thermal}]
		By Lemma~\ref{lem:thermal_moments} and the definition of the normalization constant $Z_\beta$, we get that $\op\in\cP_2(\h)$. On the other hand, Lemma~\ref{lem:thermal_gradient} gives a bound on $\sNrm{\Dh\sqrt{\op}}{\L^2}$. The proposition then follows from these two facts and  Theorem~\ref{thm:Wh_bound}.
	\end{proof}
	
% ----- Acknowledgments -----
\medskip
\paragraph{\bf Acknowledgment.} This project has received funding from the European Research Council (ERC) under the European Union’s Horizon 2020 research and innovation program (grant agreement No 865711).

%% ********************  Bibliographie  ********************

\renewcommand{\bibname}{\centerline{Bibliography}}
\bibliographystyle{abbrv} % apalike, ieee, plain, alpha, unsrt, abbrv
\bibliography{Vlasov}

\end{document}